\documentclass[a4paper,12pt]{amsart}
\usepackage{amsmath}
\usepackage{amsthm}
\usepackage{amssymb}
\usepackage{amscd}
\usepackage[all]{xy}

\newtheorem{thm}{Theorem}[section]
\newtheorem{lemma}[thm]{Lemma}

\newcommand{\proj}{\mathop{\rm Proj}\nolimits}

\DeclareMathOperator{\Hom}{Hom}
\DeclareMathOperator{\Ext}{Ext}

\DeclareMathOperator{\RHom}{RHom}
\DeclareMathOperator{\lotimes}{\otimes^{L}}
\newcommand{\caltor}{\mathop{{\mathcal T\!or}}\nolimits}

\DeclareMathOperator{\End}{End}

\DeclareMathOperator{\rk}{rank}
\DeclareMathOperator{\module}{mod}

\DeclareMathOperator{\Gr}{Gr}

\makeatletter

\@addtoreset{equation}{section}
\makeatother

\title[
Nef vector bundles on a hyperquadric
]{
Nef vector bundles on a hyperquadric
with first Chern class two
}

\thanks{
This work was partially supported by 
JSPS KAKENHI (C) Grant Number 21K03158.
}

\author{Masahiro Ohno
}

\address{Graduate School of Informatics and Engineering,
The University of Electro-Communications,
Chofu-shi,
Tokyo, 182-8585 Japan
}

\email{masahiro-ohno@uec.ac.jp}



\begin{document}
\begin{abstract}
We classify nef vector bundles on a smooth hyperquadric
of dimension $\geq 4$ 
with first Chern class two 
over an algebraically closed field of characteristic zero.
\end{abstract}

\maketitle


\section{Introduction}
Peternell-Szurek-Wi\'{s}niewski
classified
nef vector bundles on a smooth hyperquadric $\mathbb{Q}^n$ of dimension $n\geq 3$ 
with first Chern class $\leq 1$
over an algebraically closed field $K$ of characteristic zero
in \cite[\S~2 Theorem~2]{pswnef},
and 
we provided
a different proof of this classification
in \cite[Theorem~9.3]{MR4453350}, 
proof which was based on an analysis with a full strong exceptional collection
of vector bundles on $\mathbb{Q}^n$.

This paper is a continuation of 
the previous paper \cite[Theorem 1.1]{QuadricThreefoldc1=2},
in which 
we have classified nef vector bundles on a smooth quadric threefold 
$\mathbb{Q}^3$
with first Chern class two.
In this paper, we classify those on 
a smooth hyperquadric $\mathbb{Q}^n$  of dimension $n\geq 4$.
The precise statement is as follows.

\begin{thm}\label{Chern2OnQ4}
Let $\mathcal{E}$ be a nef vector bundle of rank $r$
on a smooth hyperquadric $\mathbb{Q}^n$ of dimension $n\geq 4$
over an algebraically closed field $K$ 
of characteristic zero,
let $\mathcal{S}$, 
$\mathcal{S}^+$, and $\mathcal{S}^-$ be spinor bundles,
and $\mathcal{S}^{\pm}$ denotes $\mathcal{S}^+$ or $\mathcal{S}^-$.
Suppose that 
$\det\mathcal{E}\cong \mathcal{O}(2)$.
Then 
$\mathcal{E}$ is isomorphic to one of the following vector bundles
or fits in one of the following exact sequences:
\begin{enumerate}
\item[(1)]
$\mathcal{O}(2)\oplus \mathcal{O}^{\oplus r-1}$;
\item[(2)] $\mathcal{O}(1)^{\oplus 2}\oplus\mathcal{O}^{\oplus r-2}$;
\item[(3)] $n=4$, 
$\mathcal{O}(1)\oplus\mathcal{S}^{\pm}
\oplus\mathcal{O}^{\oplus r-3}$;
\item[(4)] $0\to \mathcal{O}(-1)\to \mathcal{O}(1)\oplus 
\mathcal{O}^{\oplus r}
\to \mathcal{E}\to 0$;
\item[(5)] $n=4$, $0\to \mathcal{O}
\to \mathcal{S}^{\pm}\oplus\mathcal{S}^{\pm}
\oplus\mathcal{O}^{\oplus r-3}
\to \mathcal{E}\to 0$;
\item[(6)] $n=5$, 
$0\to \mathcal{O}
\to \mathcal{S}\oplus\mathcal{O}^{\oplus r-3}
\to \mathcal{E}\to 0$;
\item[(7)] $n=6$, 
$0\to \mathcal{O}
\to \mathcal{S}^{\pm}\oplus\mathcal{O}^{\oplus r-3}
\to \mathcal{E}\to 0$;
\item[(8)] $n=4$, 
$0\to \mathcal{S}^{\pm}(-1)\oplus\mathcal{O}(-1)
\to \mathcal{O}^{\oplus r+3}
\to \mathcal{E}\to 0$;
\item[(9)] $0\to \mathcal{O}(-1)^{\oplus 2}\to 
\mathcal{O}^{\oplus r+2}
\to \mathcal{E}\to 0$;
\item[(10)] $0\to \mathcal{O}(-2)\to \mathcal{O}^{\oplus r+1}
\to \mathcal{E}\to 0$.
\end{enumerate}
\end{thm}
Note that this list is effective: in each case exits an example.
Theorem~\ref{Chern2OnQ4} answers the question posed in 
\cite[Question 1]{MR3275418}.
Note that Case (8) in Theorem~\ref{Chern2OnQ4}
was missing in \cite[Example 1]{MR3275418}.
Note also that the projectivization $\mathbb{P}(\mathcal{E})$
of the bundle $\mathcal{E}$ in Theorem~\ref{Chern2OnQ4}
is a Fano manifold of dimension $n+r-1$, i.e., the bundle $\mathcal{E}$
in Theorem~\ref{Chern2OnQ4} is a Fano bundle  on $\mathbb{Q}^n$ of rank $r$.

Our basic strategy and framework for  describing  
$\mathcal{E}$
in Theorem~\ref{Chern2OnQ4} 
is to provide a minimal locally free resolution 
of $\mathcal{E}$
in terms of some twists of the full strong exceptional collection 
\[(\mathcal{O},\mathcal{S},\mathcal{O}(1),\mathcal{O}(2),\dots,\mathcal{O}(n-1))\] 
of vector bundles, if $n$ is odd; collection
\[(\mathcal{O},\mathcal{S}^+,\mathcal{S}^-,\mathcal{O}(1),\mathcal{O}(2),\dots,\mathcal{O}(n-1))\]
of vector bundles, if $n$ is even 
(see \cite{MR4453350} for more details).

The content of this paper is as follows.
In Section~\ref{Preliminaries}, we briefly recall 
Bondal's theorem \cite[Theorem 6.2]{MR992977}
and its related notions
and results
used in the proof of Theorem~\ref{Chern2OnQ4}.
In particular, we recall some finite-dimensional algebra $A$,
and fix symbols, 
e.g., $G$ and $S_i$,  
related to 
$A$ and 
finitely generated right $A$-modules. 
We also recall the classifcation
\cite[Theorem 1.1]{QuadricThreefoldc1=2}
of nef vector bundles
on a smooth quadric threefold $\mathbb{Q}^3$ with Chern class two
in Theorem~\ref{Chern2OnQ3}.
In Section~\ref{HRR}, we state Hirzebruch-Riemann-Roch formulas
for vector bundles $\mathcal{E}$ on $\mathbb{Q}^4$ with $c_1=2$.
In Section~\ref{Set-up for the case $n=4$}, 
we provide the set-up for the proof of Theorem~\ref{Chern2OnQ4}.
The proof of Theorem~\ref{Chern2OnQ4} is carried out 
in Sections~\ref{Case(1)OfTheoremChern2OnQ3} to 
\ref{Case(9)OfTheoremChern2OnQ3}
according to which case of Theorem~\ref{Chern2OnQ3} $\mathcal{E}|_{\mathbb{Q}^3}$ belongs to.

\subsection{Notation and conventions}\label{convention}
Throughout this paper,
we work over an algebraically closed field $K$ of characteristic zero.
Basically we follow the standard notation and terminology in algebraic
geometry. 
We denote 
by $\mathbb{Q}^n$ a smooth hyperquadric of dimension $n$ over $K$,
and by 
\begin{itemize}
\item $\mathcal{S}$, $\mathcal{S}^+$, $\mathcal{S}^-$ spinor bundles on $\mathbb{Q}^n$.
\end{itemize}
Note that we follow Kapranov's convention~\cite[p.\ 499]{MR0939472};
our spinor bundles $\mathcal{S}$, $\mathcal{S}^+$, $\mathcal{S}^-$ are globally generated,
and they are the duals of those of Ottaviani's~\cite{ot}.
For a coherent sheaf $\mathcal{F}$,
we denote by $c_i(\mathcal{F})$ the $i$-th Chern class of $\mathcal{F}$
and 
by $\mathcal{F}^{\vee}$ the dual of $\mathcal{F}$.
In particular, 
\begin{itemize}
\item $c_i$ stands for $c_i(\mathcal{E})$
of the nef vector bundle $\mathcal{E}$ we are dealing with.
\end{itemize}
For a vector bundle $\mathcal{E}$,
$\mathbb{P}(\mathcal{E})$ denotes $\proj S(\mathcal{E})$,
where $S(\mathcal{E})$ denotes the symmetric algebra of $\mathcal{E}$.
The tautological line bundle 
\begin{itemize}
\item $\mathcal{O}_{\mathbb{P}(\mathcal{E})}(1)$
is also denoted by $H(\mathcal{E})$.
\end{itemize}
Let $A^*\mathbb{Q}^4$ be the Chow ring of $\mathbb{Q}^4$.
We denote 
\begin{itemize}
\item by $h$ the hyperplane class in $A^1\mathbb{Q}^4$: $A^1\mathbb{Q}^4=\mathbb{Z}h$.
\end{itemize}
Finally we refer to \cite{MR2095472} for the definition
and basic properties of nef vector bundles.

\section{Preliminaries}\label{Preliminaries}
Throughout this paper, 
$G_0$, $G_1$, $G_2$, $G_3\dots,G_m$ denote respectively 
\[\mathcal{O}, \mathcal{S}, \mathcal{O}(1), \mathcal{O}(2),\dots,\mathcal{O}(n-1)\]
on $\mathbb{Q}^n$ if $n$ is odd, where $m=n$,
and 
\[\mathcal{O}, \mathcal{S}^+, \mathcal{S}^-, \mathcal{O}(1),\dots,\mathcal{O}(n-1)\]
on $\mathbb{Q}^n$ if $n$ is even, where $m=n+1$.
An important and  well known fact~\cite[Theorem~4.10]{MR0939472} of the 
collection $(G_0,\dots,G_m)$
is that it is a full strong exceptional collection in 
$D^b(\mathbb{Q}^n)$,
where $D^b(\mathbb{Q}^n)$ denotes the bounded derived category
of (the abelian category of) coherent sheaves on $\mathbb{Q}^n$.
An exceptional collection is also called an exceptional sequence. 
We refer to \cite{MR2244106} for the definition of a full strong exceptional sequence.

Denote by $G$ the direct sum $\bigoplus_{i=0}^mG_i$ of $G_0,\dots, G_m$,
and by $A$ the endomorphism ring $\End(G)$ of $G$.
Then $A$ is a finite-dimensional $K$-algebra,
and $G$ is a left $A$-module.
Denote by $\module A$ the category of finitely generated right $A$-modules,
and by $D^b(\module A)$ the bounded derived category
of $\module A$.
Let $p_i:G\to G_i$ be the projection,
and $\iota_i:G_i\hookrightarrow G$ the inclusion.
Set $e_i=\iota_i\circ p_i$. Then  $e_i\in A$.
Set $P_i=e_iA$.
Then $A\cong \oplus_i P_i$ as right $A$-modules,
and hence $P_i$ is projective.
We have $P_i\otimes_A G\cong G_i$.
Any finitely generated right $A$-module $V$
has an ascending filtration 
\[0=V^{\leq -1}\subset V^{\leq 0}\subset V^{\leq 1}\subset 
\dots \subset V^{\leq m}=V\]
by right $A$-submodules,
where $V^{\leq i}=\bigoplus_{j\leq i}Ve_j$.
Set $\Gr^iV=V^{\leq i}/V^{\leq i-1}$
and $S_i=\Gr^iP_i$.
Then $S_i$ is a simple right $A$-module.
If we set $m_i=\dim_K \Gr^iV$, then 
$\Gr^iV\cong S_i^{\oplus m_i}$.
For a coherent sheaf $\mathcal{F}$ 
on $\mathbb{Q}^n$,
$\Ext^q(G, \mathcal{F})$ is a finitely generated 
right $A$-module.

It follows from Bondal's theorem \cite[Theorem 6.2]{MR992977}
that 
\[\RHom(G,\bullet):D^b(\mathbb{Q}^n)\to D^b (\module A)\]
is an exact equivalence,
and its quasi-inverse is 
\[\bullet\lotimes_AG:D^b(\module A)\to D^b(\mathbb{Q}^n).\]
For a coherent sheaf $\mathcal{F}$ on $\mathbb{Q}^n$, this fact can be rephrased 
in terms of a spectral sequence\cite[Theorem 1]{MR3275418}:
\begin{equation}\label{BondalSpectral}
E_2^{p,q}=\caltor_{-p}^A(\Ext^q(G,\mathcal{F}),G)
\Rightarrow
E^{p+q}=
\begin{cases}
\mathcal{F}& \textrm{if}\quad  p+q= 0\\
0& \textrm{if}\quad  p+q\neq 0,
\end{cases}
\end{equation}
which we call 
the Bondal spectral sequence.
Note that $E_2^{p,q}$ is the $p$-th cohomology sheaf 
$\mathcal{H}^p(\Ext^q(G,\mathcal{F})\lotimes_AG)$ of the complex 
$\Ext^q(G,\mathcal{F})\lotimes_AG$.
When we compute the spectral sequence, we consider the descending filtration
on the right $A$-module $\Ext^q(G,\mathcal{F})$, and 
consider distinguished triangles consisting of 
the pull-backs  $\Gr^i\Ext^q(G,\mathcal{F})\lotimes_AG$ of the graded pieces 
$\Gr^i\Ext^q(G,\mathcal{F})$.
Thus it is important to know 
what the pull-back $S_i\lotimes_AG$ of the component $S_i$ of the graded piece 
$\Gr^i\Ext^q(G,\mathcal{F})$ is.

Here we recall the bundles $\Psi_i$ on $\mathbb{Q}^n$,
which is defined by Kapranov in \cite[\S 4]{MR0939472},
and which is characterized by Ancona and Ottaviani in \cite[\S 5]{MR1092580}
by the non-split exact sequences
\[0\to \Omega_{\mathbb{P}^{n+1}}^i(i)|_{\mathbb{Q}^n}
\to \Psi_i\to \Psi_{i-2}\to 0\]
for $i\geq 2$ with $\Psi_0=\mathcal{O}$ and $\Psi_1=\Omega_{\mathbb{P}^{n+1}}^1(1)|_{\mathbb{Q}^n}$.
Set 
\[\widetilde{\Psi}_i
=\Psi_{n+1-i}^{\vee}\otimes \Omega_{\mathbb{P}^{n+1}}^{n+1}(n+1)|_{\mathbb{Q}^n}.\]
Then 
$\widetilde{\Psi}_{n+1}
=\mathcal{O}(-1)$,
$\widetilde{\Psi}_{n}
=T_{\mathbb{P}^{n+1}}(-2)|_{\mathbb{Q}^n}
\cong \Omega_{\mathbb{P}^{n+1}}^{n}(n)|_{\mathbb{Q}^n}$,
and, 
for $2\leq i\leq n+1$,
$\widetilde{\Psi}_{i}$ fits in the following non-split exact sequence:
\begin{equation}\label{PsiTildeExactSeq}
0\to \widetilde{\Psi}_{i}\to \widetilde{\Psi}_{i-2}\to 
\Omega_{\mathbb{P}^{n+1}}^{i-2}(i-2)|_{\mathbb{Q}^n}
\to 0.
\end{equation}
First we note 
the following 
\begin{lemma}\label{SpinorVanishOmegaP}
Let $\mathcal{S}^{(\pm)}$ denotes $\mathcal{S}^+$ or $\mathcal{S}^-$ if $n$ is even
and $\mathcal{S}$ if $n$ is odd. Then 
\begin{enumerate}
\item $\RHom(\mathcal{O}(k), \mathcal{S}^{(\pm)}(-1))=0$
for $0\leq k\leq n-1$.
\item 
$\RHom(\mathcal{S}^{(\pm)},
\widetilde{\Psi}_p
)=0$
 for $3\leq p\leq n+1$.
\end{enumerate}
\end{lemma}
\begin{proof}
(1) This follows from \cite[Proposition 4.11]{MR0939472}.

(2) Recall that
$\mathcal{S}^{(\pm)\vee}\cong \mathcal{S}^{(\pm)}(-1)$ 
(see, e.g.,
\cite[Theorem 8.1 (9), (10), (11)]{MR4453350}).
Hence we have 
$\RHom(\mathcal{S}^{(\pm)},\mathcal{O}(-k))=0$ for $0\leq k\leq n-1$
by (1).
In particular, $\RHom(\mathcal{S}^{(\pm)},\widetilde{\Psi}_{n+1})=0$.
Set $V=H^0(\mathcal{O}_{\mathbb{P}^{n+1}}(1))$.
Now 
$\Omega_{\mathbb{P}^{n+1}}^{p}(p)|_{\mathbb{Q}^n}$
is isomorphic to 
$\wedge^{n+1-p}(T_{\mathbb{P}^{n+1}}(-1)|_{\mathbb{Q}^n})(-1)$,
and we have the following exact sequence: 
\[0\to \mathcal{O}(-1)\to \mathcal{O}\otimes V^{\vee}\to 
T_{\mathbb{P}^{n+1}}(-1)|_{\mathbb{Q}^n}\to 0.\]
This exact sequence induces the following exact sequence:
\[0\to \wedge^{k-1}(T_{\mathbb{P}^{n+1}}(-1)|_{\mathbb{Q}^n})(-1)
\to \mathcal{O}\otimes \wedge^kV^{\vee}\to 
\wedge^{k}(T_{\mathbb{P}^{n+1}}(-1)|_{\mathbb{Q}^n})\to 0.\]
Hence 
$\wedge^{k}(T_{\mathbb{P}^{n+1}}(-1)|_{\mathbb{Q}^n})(-1)$
has the following resolution:
\[0\to \mathcal{O}(-k-1)\to \mathcal{O}(-k)\otimes V^{\vee}\to \dots
\to \mathcal{O}(-1)\otimes \wedge^{k}V^{\vee}\to 
\wedge^{k}(T_{\mathbb{P}^{n+1}}(-1)|_{\mathbb{Q}^n})(-1)\to 0.\]
Therefore $\RHom(\mathcal{S}^{(\pm)},\wedge^{k}(T_{\mathbb{P}^{n+1}}(-1)|_{\mathbb{Q}^n})(-1))=0$ 
for $0\leq k\leq n-2$. 
Hence we see that $\RHom(\mathcal{S}^{(\pm)},\widetilde{\Psi}_{n})=0$
and that $\RHom(\mathcal{S}^{(\pm)},\Omega_{\mathbb{P}^{n+1}}^p(p)|_{\mathbb{Q}^n})=0$
for $3\leq p\leq n+1$.
Now the exact sequence \eqref{PsiTildeExactSeq} inductively shows  that 
$\RHom(\mathcal{S}^{(\pm)},\widetilde{\Psi}_p)=0$
for $3\leq p\leq n+1$.
\end{proof}
Now we can answer the question that what $S_i\lotimes_AG$ is.
\begin{lemma}\label{S2Arilemma}
\begin{enumerate}
\item[(1)]
If $n$ is odd,
we have 
\begin{align}
S_{p-1}\lotimes_AG&\cong \widetilde{\Psi}_{p}[p-1] 
\textrm{ for } n+1\geq p\geq 3; \label{SiOdd}
\\ 
S_1\lotimes_AG&
\cong \mathcal{S}(-1)[1];
\label{S1Sdual}\\
S_0\lotimes_AG&\cong \mathcal{O}.\label{S0OOdd}
\end{align}
\item[(2)] If $n$ is even, we have 
\begin{align}
S_{p}\lotimes_AG&\cong \widetilde{\Psi}_p[p-1] 
\textrm{ for } n+1\geq p\geq 3; \label{SiEven}
\\ 
S_2\lotimes_AG&\cong \mathcal{S}^{+}(-1)[1];
\label{S2S+Even}
\\
S_1\lotimes_AG&
\cong \mathcal{S}^{-}(-1)[1];
\label{S1S-Even}\\
S_0\lotimes_AG&\cong \mathcal{O}.\label{S0OEven}
\end{align}
\end{enumerate}
\end{lemma}
\begin{proof}
It follows from \cite[Proposition 4.11]{MR0939472}
that for $0\leq j, k\leq n-1$ and $0\leq q\leq n$ we have 
$h^j(\Psi_j(-j))=1$ and $h^q(\Psi_j(-k))=0$
unless $(q,k)=(j,j)$.
Since $h^q(\widetilde{\Psi}_p(-k))=h^q(\Psi_{n+1-p}^{\vee}(-1-k))=
h^{n-q}(\Psi_{n+1-p}(1+k-n))$, this implies that for $2\leq p\leq n+1$ and $0\leq k\leq n-1$
we have  
\begin{equation}\label{KapranovVanishingCor}
h^{p-1}(\widetilde{\Psi}_p(-p+2))=1
\textrm{ and }h^q(\widetilde{\Psi}_p(-k))=0\textrm{ unless }(q,k)=(p-1,p-2).
\end{equation}

Suppose that $n$ is odd. Then \eqref{KapranovVanishingCor} 
and Lemma~\ref{SpinorVanishOmegaP} (2)
show that $\RHom(G,\widetilde{\Psi}_p)=S_{p-1}[1-p]$ for $3\leq p\leq n+1$.
Therefore we obtain \eqref{SiOdd}.
Suppose that $n$ is even. 
Then \eqref{KapranovVanishingCor}  and Lemma~\ref{SpinorVanishOmegaP} (2)
show that $\RHom(G,\widetilde{\Psi}_p)=S_{p}[1-p]$ for $3\leq p\leq n+1$.
Therefore we obtain \eqref{SiEven}.

Suppose that $n$ is odd. Then it follows from \cite[Lemma 8.2 (1)]{MR4453350}
and Lemma~\ref{SpinorVanishOmegaP} (1) that $\RHom(G,\mathcal{S}(-1))=S_1[-1]$.
Hence we obtain \eqref{S1Sdual}.
Suppose that $n$ is even. Then 
it follows from \cite[Lemma 8.2 (1)]{MR4453350}
and Lemma~\ref{SpinorVanishOmegaP} (1) that $\RHom(G,\mathcal{S}^{-}(-1))=S_1[-1]$
and $\RHom(G,\mathcal{S}^{+}(-1))=S_2[-1]$.
Hence we obtain \eqref{S1S-Even} and \eqref{S2S+Even}.

Finally \eqref{S0OOdd} and \eqref{S0OEven} hold since $S_0$ is nothing but 
a projective module $P_0=e_0A$.
\end{proof}

Theorem~\ref{Chern2OnQ4} will be proved, based on the following
theorem~\cite[Theorem 1.1]{QuadricThreefoldc1=2}:
\begin{thm}\label{Chern2OnQ3}
Let $\mathcal{E}$ be a nef vector bundle of rank $r$
on a smooth hyperquadric $\mathbb{Q}^3$ of dimension $3$
over an algebraically closed field $K$ 
of characteristic zero,
and let $\mathcal{S}$ be the spinor bundle on $\mathbb{Q}^3$.
Suppose that 
$\det\mathcal{E}\cong \mathcal{O}(2)$.
Then 
$\mathcal{E}$ is isomorphic to one of the following vector bundles
or fits in one of the following exact sequences:
\begin{enumerate}
\item[(1)]
$\mathcal{O}(2)\oplus \mathcal{O}^{\oplus r-1}$;
\item[(2)] $\mathcal{O}(1)^{\oplus 2}\oplus\mathcal{O}^{\oplus r-2}$;
\item[(3)] 
$\mathcal{O}(1)\oplus 
\mathcal{S}\oplus
\mathcal{O}^{\oplus r-3}$;
\item[(4)] $0\to \mathcal{O}(-1)\to \mathcal{O}(1)\oplus 
\mathcal{O}^{\oplus r}
\to \mathcal{E}\to 0$;
\item[(5)] $0\to \mathcal{O}^{\oplus a}\to\mathcal{S}^{\oplus 2}\oplus\mathcal{O}^{\oplus r-4+a}
\to \mathcal{E}\to 0$, where $a=0$ or $1$,
and the composite of the injection 
$\mathcal{O}^{\oplus a}\to\mathcal{S}^{\oplus 2}\oplus\mathcal{O}^{\oplus r-4+a}$
and the projection $\mathcal{S}^{\oplus 2}\oplus\mathcal{O}^{\oplus r-4+a}
\to \mathcal{O}^{\oplus r-4+a}$ is zero;
\item[(6)] 
$0\to \mathcal{S}(-1)\oplus  \mathcal{O}(-1)
\to \mathcal{O}^{\oplus r+3}
\to \mathcal{E}\to 0$;
\item[(7)] $0\to \mathcal{O}(-1)^{\oplus 2}\to 
\mathcal{O}^{\oplus r+2}
\to \mathcal{E}\to 0$;
\item[(8)] $0\to \mathcal{O}(-2)\to \mathcal{O}^{\oplus r+1}
\to \mathcal{E}\to 0$;
\item[(9)] 
$0\to \mathcal{O}(-2)\to 
\mathcal{O}(-1)^{\oplus 4}
\to 
\mathcal{O}^{\oplus r+3}
\to \mathcal{E}\to 0$.
\end{enumerate}
\end{thm}

\section{Hirzebruch-Riemann-Roch formulas}\label{HRR}
Let $\mathcal{E}$ be a vector bundle of rank $r$ on $\mathbb{Q}^4$.
Since the tangent bundle $T$ of $\mathbb{Q}^4$ fits in an exact sequence
\[0\to T\to T_{\mathbb{P}^5}|_{\mathbb{Q}^4}
\to \mathcal{O}_{\mathbb{Q}^4}(2)\to 0,\]
the Chern polynomial $c_t(T)$ of $T$ is 
\[\dfrac{(1+ht)^6}{1+2ht}=1+4ht+7h^2t^2+6h^3t^3+3h^4t^4,\]
where $h$ denotes $c_1(\mathcal{O}_{\mathbb{Q}^4}(1))$.
Then the Hirzebruch-Riemann-Roch formula implies that
\[
\begin{split}
\chi(\mathcal{E})=&r+\dfrac{7}{6}c_1h^3
+\dfrac{23}{24}(c_1^2-2c_2)h^2+
\dfrac{1}{3}(c_1^3-3c_1c_2
+3c_3)h\\
&+\dfrac{1}{24}(c_1^4-4c_1^2c_2+4c_1c_3+2c_2^2-4c_4),
\end{split}
\]
where we set $c_i=c_i(\mathcal{E})$.
To compute $\chi(\mathcal{E}(t))$, note that 
\begin{align*}
c_1(\mathcal{E}(t))&=c_1+rht;\\
c_2(\mathcal{E}(t))&
=c_2+(r-1)c_1ht+
\binom{r}{2}
h^2t^2;\\
c_3(\mathcal{E}(t))&=c_3+(r-2)c_2ht+
\binom{r-1}{2}c_1h^2t^2+\binom{r}{3}h^3t^3;\\
c_4(\mathcal{E}(t))&=c_4+(r-3)c_3ht+
\binom{r-2}{2}c_2h^2t^2+\binom{r-1}{3}h^3t^3+\binom{r}{4}h^4t^4.
\end{align*}
Since $h^4=2$, we infer that 
\begin{equation}\label{generalRRonQ3}
\begin{split}
\chi(\mathcal{E}(t))=&\dfrac{r}{12}(t+1)(t+2)^2(t+3)+\dfrac{1}{6}c_1h^3t^3
+\dfrac{1}{4}\{4c_1h
+(c_1^2-2c_2)\}h^2t^2\\
&+\dfrac{1}{12}\{23c_1h^2+12(c_1^2-2c_2)h+2(c_1^3-3c_1c_2+3c_3)\}ht\\
&+\dfrac{7}{6}c_1h^3
+\dfrac{23}{24}(c_1^2-2c_2)h^2+
\dfrac{1}{3}(c_1^3-3c_1c_2
+3c_3)h\\
&+\dfrac{1}{24}(c_1^4-4c_1^2c_2+4c_1c_3+2c_2^2-4c_4).
\end{split}
\end{equation}
Since $c_1(\mathcal{E})=dh$ for some integer $d$,
the formula above can be written as
\begin{equation}
\begin{split}
\chi(\mathcal{E}(t))=&\dfrac{r}{12}(t+1)(t+2)^2(t+3)+\dfrac{d}{3}t^3
+\dfrac{1}{2}\{4d
+(d^2-c_2h^2)\}t^2\\
&+\dfrac{1}{6}\{23d+12(d^2-c_2h^2)+(2d^3-3dc_2h^2+3c_3h)\}t\\
&+\dfrac{7}{3}d
+\dfrac{23}{12}(d^2-c_2h^2)+
\dfrac{1}{3}(2d^3-3dc_2h^2
+3c_3h)\\
&+\dfrac{1}{12}(d^4-2d^2c_2h^2+2dc_3h+c_2^2-2c_4).
\end{split}
\end{equation}
In this paper, we are dealing with the case $d=2$:
\begin{equation}\label{e(t)RR}
\begin{split}
\chi(\mathcal{E}(t))=&\dfrac{r}{12}(t+1)(t+2)^2(t+3)+\dfrac{1}{3}(2t^3+18t^2+55t+57)
-\dfrac{1}{2}c_2h^2t^2\\
&+\dfrac{1}{2}(c_3h-6c_2h^2)t+\dfrac{1}{12}(16c_3h-55c_2h^2)
+\dfrac{1}{12}(c_2^2-2c_4).
\end{split}
\end{equation}
In particular,
\begin{align}
\chi(\mathcal{E}(-1))&=
6-\dfrac{25}{12}c_2h^2
+\dfrac{5}{6}c_3h+\dfrac{1}{12}(c_2^2-2c_4);\label{e(-1)RR}\\
\chi(\mathcal{E}(-2))&=
1-\dfrac{7}{12}c_2h^2
+\dfrac{1}{3}c_3h++\dfrac{1}{12}(c_2^2-2c_4).\label{e(-2)RR}
\end{align}

Next we will compute $\chi(\mathcal{S}^{\pm\vee}\otimes \mathcal{E}(t))$.
Recall that $c_1(\mathcal{S}^{\pm})=h$.
Note also that 
\begin{align*}
\rk \mathcal{S}^{\pm\vee}\otimes \mathcal{E}=&2r;\\
c_1(\mathcal{S}^{\pm\vee}\otimes \mathcal{E})=&2c_1-rh;\\
c_2(\mathcal{S}^{\pm\vee}\otimes \mathcal{E})=
&2c_2-(2r-1)c_1h
+c_1^2+\binom{r}{2}h^2+rc_2(\mathcal{S}^{\pm});\\
c_3(\mathcal{S}^{\pm\vee}\otimes\mathcal{E})=
&2c_3-2(r-1)c_2h+(r-1)^2c_1h^2
+2(r-1)c_1c_2(\mathcal{S}^{\pm})\\
&+2c_1c_2-(r-1)c_1^2h-\binom{r}{3}h^3-r(r-1)c_2(\mathcal{S}^{\pm})h;\\
c_4(\mathcal{S}^{\pm\vee}\otimes\mathcal{E})=
&2c_4+2c_1c_3+c_2^2-(2r-3)c_3h+(r^2-3r+3)c_2h^2
+2(r-3)c_2c_2(\mathcal{S}^{\pm})\\
&-\dfrac{1}{6}(r-1)(r-2)(2r-3)c_1h^3-(r-1)(2r-3)c_1c_2(\mathcal{S})h\\
&-(2r-3)c_1c_2h+\binom{r-1}{2}c_1^2h^2+(r-1)c_1^2c_2(\mathcal{S})+\dfrac{1}{12}r^2(r^2-1).
\end{align*}
Since $c_2(\mathcal{S}^{\pm})^2=1$ and $c_2(\mathcal{S}^{\pm})h^2=1$,
the formula \eqref{generalRRonQ3} together with the formulas above 
implies the following formula:
\begin{align*}
\chi(\mathcal{S}^{\pm\vee}\otimes \mathcal{E}(t))=&\dfrac{r}{6}t(t+1)(t+2)(t+3)+\dfrac{1}{3}c_1h^3t^3
+\dfrac{1}{2}(3c_1h^3+c_1^2h^2-2c_2h^2)t^2\\
&+\dfrac{1}{6}(14c_1h^3+9c_1^2h^2-18c_2h^2+2c_1^3h-6c_1c_2(\mathcal{S}^{\pm})h-6c_1c_2h+6c_3h)t
\\
&+\dfrac{1}{12}(15c_1h^3+14c_1^2h^2-28c_2h^2+6c_1^3h-18c_1c_2(\mathcal{S}^{\pm})h-18c_1c_2h+18c_3h)
\\
&+\dfrac{1}{12}(c_1^4-4c_1^2c_2+4c_1c_3+2c_2^2-4c_4)
-\dfrac{1}{2}c_1^2c_2(\mathcal{S}^{\pm})+c_2c_2(\mathcal{S}^{\pm}).
\end{align*}
Since $c_1(\mathcal{E})=dh$, the formula above becomes the following formula:
\begin{equation}
\begin{split}\label{SobokuRRwithSpinor}
\chi(\mathcal{S}^{\pm\vee}\otimes \mathcal{E}(t))=&
\dfrac{r}{6}t(t+1)(t+2)(t+3)+\dfrac{2}{3}dt^3+(d^2+3d-c_2h^2)t^2\\
&+\dfrac{1}{3}\{2d^3+9d^2+11d-3(d+3)c_2h^2+3c_3h\}t\\
&+\dfrac{1}{6}d(d+1)(d+2)(d+3)
-\dfrac{1}{6}(2d^2+9d+14)c_2h^2\\
&+\dfrac{1}{6}(2d+9)c_3h+\dfrac{1}{6}(c_2^2-2c_4)
+c_2c_2(\mathcal{S}^{\pm}).
\end{split}
\end{equation}
In this paper, we are dealing with the case $d=2$:
\begin{equation}
\begin{split}\label{SERR}
\chi(\mathcal{S}^{\pm\vee}\otimes \mathcal{E}(t))=&
\dfrac{r}{6}t(t+1)(t+2)(t+3)+\dfrac{2}{3}(2t+5)(t+2)(t+3)\\
&-\left(t^2+5t+\dfrac{20}{3}\right)c_2h^2+\left(t+\dfrac{13}{6}\right)c_3h\\
&+\dfrac{1}{6}(c_2^2-2c_4)+c_2c_2(\mathcal{S}^{\pm}).
\end{split}
\end{equation}
In particular, we have the following:
\begin{equation}\label{SERR(-1)}
\chi(\mathcal{S}^{\pm\vee}\otimes \mathcal{E}(-1))=
4-\dfrac{8}{3}c_2h^2+\dfrac{7}{6}c_3h+\dfrac{1}{6}(c_2^2-2c_4)+c_2c_2(\mathcal{S}^{\pm}).
\end{equation}

\section{Set-up for the proof of Theorem~\ref{Chern2OnQ4}
}\label{Set-up for the case $n=4$}
Let
$\mathcal{E}$ be a nef vector bundle of rank $r$ on $\mathbb{Q}^n$
of dimension $n\geq 4$
with 
$c_1=2h$.
If $h^0(\mathcal{E}(-2))\neq 0$,
then $\mathcal{E}\cong 
\mathcal{O}(2)\oplus \mathcal{O}^{\oplus r-1}$ by 
\cite[Proposition 5.1 and Remark 5.3]{MR4453350}.
This is Case (1) of Theorem~\ref{Chern2OnQ4}.
Thus we will always assume 
that 
\begin{equation}\label{h^0(E(-2))vanish}
h^0(\mathcal{E}(-2))=0
\end{equation}
in the following.
It follows from \cite[Lemma~4.1 (1)]{MR4453350} that
\begin{equation}\label{firstvanishing}
h^q(\mathcal{E}(t))=0 \textrm{ for } q>0 \textrm{ and } t\geq -1.
\end{equation}
Furthermore, from \eqref{h^0(E(-2))vanish} and \cite[Lemma~4.1 (1)]{MR4453350}, 
it follows that 
\begin{equation}\label{firstvanishinggeq5}
h^q(\mathcal{E}(-2))=0 \textrm{ for any } q, \textrm{ if }n\geq 5.
\end{equation}
Denote by $(\mathcal{S}^{(\pm)},\mathcal{S}^{(\mp)})$
a pair $(\mathcal{S}^{+},\mathcal{S}^{-})$
or $(\mathcal{S}^{-},\mathcal{S}^{+})$ of spinor bundles if $n$ is even,
and a pair $(\mathcal{S},\mathcal{S})$ of the spinor bundle if $n$ is odd.
Ottaviani shows in \cite[Theorem~2.8]{ot} that 
we have an exact sequence
\begin{equation}\label{SSdual}
0\to \mathcal{S}^{(\pm)\vee}\to \mathcal{O}^{\oplus \alpha}\to 
\mathcal{S}^{(\mp)\vee}(1)\to 0,
\end{equation}
where $\alpha=2^{\lceil n/2\rceil}$.
Hence we have the following exact sequences:
\begin{equation}\label{E(-2)twistofSSdual}
0\to \mathcal{S}^{(\pm)\vee}\otimes \mathcal{E}(-2)\to 
\mathcal{E}(-2)^{\oplus \alpha}\to \mathcal{S}^{(\mp)\vee}\otimes\mathcal{E}(-1)\to 0;
\end{equation}
\begin{equation}\label{E(-1)twistofSSdual}
0\to \mathcal{S}^{(\pm)\vee}\otimes \mathcal{E}(-1)\to 
\mathcal{E}(-1)^{\oplus \alpha}\to \mathcal{S}^{(\mp)\vee}\otimes\mathcal{E}\to 0.
\end{equation}
The exact sequence~\eqref{E(-2)twistofSSdual}  
together with \eqref{firstvanishinggeq5}
implies that 
\begin{equation}\label{qq+1joushouForE(-2)n=5}
h^q(\mathcal{S}^{(\mp)\vee}\otimes\mathcal{E}(-1))
=h^{q+1}(\mathcal{S}^{(\pm)\vee}\otimes\mathcal{E}(-2))
\textrm{ for any }q, \textrm{ if }n\geq 5.
\end{equation}
The exact sequence~\eqref{E(-1)twistofSSdual} together with 
\eqref{firstvanishing} implies that 
\begin{equation}\label{qq+1joushouForE(-1)}
h^q(\mathcal{S}^{(\mp)\vee}\otimes\mathcal{E})
=h^{q+1}(\mathcal{S}^{(\pm)\vee}\otimes\mathcal{E}(-1))
\textrm{ for any }q, \textrm{ if }h^0(\mathcal{E}(-1))=0.
\end{equation}
In order to compute $\Ext^q(\mathcal{S}^{(\mp)},\mathcal{E}(-1))$,
we will use the following exact sequence together with 
\eqref{E(-2)twistofSSdual} or \eqref{qq+1joushouForE(-2)n=5}:
\begin{equation}\label{RestrictionForSE(-1)}
0\to \mathcal{S}^{(\pm)\vee}\otimes \mathcal{E}(-2)
\to \mathcal{S}^{(\pm)\vee}\otimes \mathcal{E}(-1)
\to (\mathcal{S}^{(\pm)\vee}\otimes \mathcal{E}(-1))|_{\mathbb{Q}^{n-1}}
\to 0.
\end{equation}
Similarly, in order to compute $\Ext^q(\mathcal{S}^{(\mp)},\mathcal{E})$,
we will use 
the following exact sequence together with 
\eqref{E(-1)twistofSSdual} or \eqref{qq+1joushouForE(-1)}:
\begin{equation}\label{RestrictionForSE}
0\to \mathcal{S}^{(\pm)\vee}\otimes \mathcal{E}(-1)
\to \mathcal{S}^{(\pm)\vee}\otimes \mathcal{E}
\to (\mathcal{S}^{(\pm)\vee}\otimes \mathcal{E})|_{\mathbb{Q}^{n-1}}
\to 0.
\end{equation}
Finally, recall that $\mathcal{S}^{\pm}|_{\mathbb{Q}^{n-1}}\cong \mathcal{S}$
if $n$ is even and that 
$\mathcal{S}|_{\mathbb{Q}^{n-1}}\cong \mathcal{S}^+\oplus \mathcal{S}^-$
if $n$ is odd (see, e.g., \cite[Theorem 8.1 (1) and (3)]{MR4453350}).

\subsection{The case $n=4$}
It follows from \cite[Lemma~4.1 (2)]{MR4453350}
that, if 
\[H(\mathcal{E})^{r+3}
=c_2^2-c_4+2c_1c_3-3c_1^2c_2+c_1^4
=c_2^2-c_4+4c_3h-12c_2h^2+32>0,\]
then 
\[
h^q(\mathcal{E}(-2))=0 \textrm{ for } q>0.
\]
Let $\Delta_{\lambda}(\mathcal{E})$ be the Schur polynomial of $\mathcal{E}$
(see, e.g., \cite[Example 12.1.7]{fl}).
Note here that 
\[
c_4=\Delta_{(4,0,0,0)}(\mathcal{E})\geq 0;\quad 
c_1c_3-c_4=\Delta_{(3,1,0,0)}(\mathcal{E})\geq 0;
\quad 
c_2^2-c_1c_3=\Delta_{(2,2,0,0)}(\mathcal{E})\geq 0,
\]
since $\mathcal{E}$ is nef (see, e.g., \cite[Chapter 8]{MR2095472}).
Hence we have
\begin{equation}\label{c_2squarebiggerC1C3biggerC4}
c_2^2\geq c_1c_3\geq c_4\geq 0.
\end{equation}
Therefore 
\[
H(\mathcal{E})^{r+3}
=c_2^2-c_4+4c_3h-12c_2h^2+32
\geq 
4c_3h-12c_2h^2+32.
\]
Thus we see that 
\begin{equation}\label{c2h<4vanishing}
h^q(\mathcal{E}(-2))=0 \textrm{ for } q>0
\textrm{ if }
c_3h-3c_2h^2+8>0.
\end{equation}
Hence it follows from 
\eqref{E(-2)twistofSSdual},
\eqref{h^0(E(-2))vanish},
and \eqref{c2h<4vanishing}
that 
\begin{equation}\label{qq+1joushouForE(-2)n=4}
h^q(\mathcal{S}^{(\mp)\vee}\otimes\mathcal{E}(-1))
=h^{q+1}(\mathcal{S}^{(\pm)\vee}\otimes\mathcal{E}(-2))
\textrm{ for any }q, \textrm{ if }
c_3h-3c_2h^2+8>0.
\end{equation}

\section{The case where $\mathcal{E}|_{\mathbb{Q}^3}$ belongs to Case
(1) of Theorem~\ref{Chern2OnQ3}}\label{Case(1)OfTheoremChern2OnQ3}
This case does not arise under the assumption~\eqref{h^0(E(-2))vanish}.
Indeed, 
if  $\mathcal{E}|_{\mathbb{Q}^3}\cong \mathcal{O}(2)
\oplus \mathcal{O}^{\oplus r-1}$ and $n=4$,
we have $c_2h^2=0$ and $c_3h=0$. Hence 
$h^q(\mathcal{E}(-2))=0$ for any $q$
by \eqref{c2h<4vanishing} and \eqref{h^0(E(-2))vanish},
and we obtain 
$\chi(\mathcal{E}(-2))=0$.
From \eqref{e(-2)RR},
it follows that 
\begin{equation}\label{c2^2-2c4negative}
\frac{1}{12}(c_2^2-2c_4)=-1.
\end{equation}

We have 
$h^q(\mathcal{S}^{\mp\vee}\otimes\mathcal{E}(-1))
=h^{q+1}(\mathcal{S}^{\pm\vee}\otimes\mathcal{E}(-2))$ for any $q$
by \eqref{qq+1joushouForE(-2)n=4}.
In particular, $h^0(\mathcal{S}^{\pm\vee}\otimes\mathcal{E}(-2))=0$.
Since $\mathcal{S}^{\vee}\otimes \mathcal{E}(-1)|_{\mathbb{Q}^3}\cong 
\mathcal{S}\oplus \mathcal{S}(-2)^{\oplus r-1}$,
we have $h^0(\mathcal{S}^{\vee}\otimes \mathcal{E}(-1)|_{\mathbb{Q}^3})=4$
and $h^q(\mathcal{S}^{\vee}\otimes \mathcal{E}(-1)|_{\mathbb{Q}^3})=0$ for $q>0$.
The exact sequence~\eqref{RestrictionForSE(-1)}
then implies that  
$h^{q+1}(\mathcal{S}^{\pm\vee}\otimes\mathcal{E}(-2))=
h^{q+1}(\mathcal{S}^{\pm\vee}\otimes\mathcal{E}(-1))$ for $q\geq 1$.
Hence we have $h^q(\mathcal{S}^{\mp\vee}\otimes\mathcal{E}(-1))
=h^{q+1}(\mathcal{S}^{\pm\vee}\otimes\mathcal{E}(-1))$ for $q\geq 1$. Consequently 
we see that $h^q(\mathcal{S}^{\mp\vee}\otimes\mathcal{E}(-1))=0$ for $q\geq 1$.
Set $a=h^0(\mathcal{S}^{+}\otimes \mathcal{E}(-1))$.
Then 
$4-a=h^1(\mathcal{S}^{+\vee}\otimes \mathcal{E}(-2))
=h^0(\mathcal{S}^{-\vee}\otimes \mathcal{E}(-1))$.
Therefore 
it follows from \eqref{SERR(-1)} and \eqref{c2^2-2c4negative} 
that 
\[a=\chi(\mathcal{S}^{+}\otimes \mathcal{E}(-1))=2+c_2c_2(\mathcal{S}^+)\]
and 
that 
\[4-a=\chi(\mathcal{S}^{-}\otimes \mathcal{E}(-1))=2+c_2c_2(\mathcal{S}^-).\]
Since $c_2(\mathcal{S}^{\pm})\cap [\mathbb{Q}^4]$ is represented by a $2$-plane $\mathbb{P}^2$,
the nefness of $\mathcal{E}|_{\mathbb{P}^2}$ implies that $c_2c_2(\mathcal{S}^{\pm})\geq 0$
by \cite[Theorem 8.2.1]{MR2095472}.
Hence we conclude that $a=2$ and $c_2c_2(\mathcal{S}^{\pm})=0$.
Set $\sigma=c_2(\mathcal{S}^{+})\cap [\mathbb{Q}^4]$
and $\tau=c_2(\mathcal{S}^{-})\cap [\mathbb{Q}^4]$. 
Then $A^2\mathbb{Q}^4$ is generated by $\sigma$ and $\tau$.
Since $c_2\sigma=0$ and $c_2\tau=0$, it follows from $c_2\in A^2\mathbb{Q}^4$ that $c_2^2=0$.
Hence  we have $c_4=0$ by \eqref{c_2squarebiggerC1C3biggerC4}.
This however contradicts \eqref{c2^2-2c4negative}.
Therefore this case does not arise under the assumption \eqref{h^0(E(-2))vanish}.

\section{The case where $\mathcal{E}|_{\mathbb{Q}^3}$ belongs to Case
(2) of Theorem~\ref{Chern2OnQ3}}\label{Case(2)OfTheoremChern2OnQ3}

\subsection{The case $n=4$}
Suppose that 
$\mathcal{E}|_{\mathbb{Q}^3}\cong 
\mathcal{O}(1)^{\oplus 2}
\oplus \mathcal{O}^{\oplus r-2}$.
Then $c_2h^2=2$ and $c_3h=0$.
Hence $h^q(\mathcal{E}(-2))=0$ for 
any $q$
by \eqref{c2h<4vanishing}
and \eqref{h^0(E(-2))vanish}.
Since $h^0(\mathcal{E}(-1)|_{\mathbb{Q}^3})=2$
and $h^q(\mathcal{E}(-1)|_{\mathbb{Q}^3})=0$ for $q>0$,
this implies that 
$h^0(\mathcal{E}(-1))=2$
and that $h^q(\mathcal{E}(-1))=0$ for $q>0$.
Since $h^q(\mathcal{E}(-2)|_{\mathbb{Q}^3})=0$ for any $q$, we infer that 
$h^q(\mathcal{E}(-3))=0$
for any $q$.
Since $h^3(\mathcal{E}(-3)|_{\mathbb{Q}^3})=r-2$
and $h^q(\mathcal{E}(-3)|_{\mathbb{Q}^3})=0$ unless $q=3$, we infer that 
$h^4(\mathcal{E}(-4))=r-2$ and that $h^q(\mathcal{E}(-4))=0$
unless $q=4$.

We have 
$h^q(\mathcal{S}^{\mp\vee}\otimes\mathcal{E}(-1))=
h^{q+1}(\mathcal{S}^{\pm\vee}\otimes\mathcal{E}(-2))
$ for any $q$
by \eqref{qq+1joushouForE(-2)n=4}.
Since 
$(\mathcal{S}^{\pm\vee}\otimes\mathcal{E}(-1))|_{\mathbb{Q}^3}
\cong \mathcal{S}(-1)^{\oplus 2}
\oplus \mathcal{S}(-2)^{\oplus r-2}$,
we see that 
$h^q((\mathcal{S}^{\pm\vee}\otimes\mathcal{E}(-1))|_{\mathbb{Q}^3})=0$
for any $q$.
The exact sequence~\eqref{RestrictionForSE(-1)}
then implies that 
$h^{q+1}(\mathcal{S}^{\pm\vee}\otimes\mathcal{E}(-2))
=h^{q+1}(\mathcal{S}^{\pm\vee}\otimes\mathcal{E}(-1))$
for any $q$.
Hence $h^q(\mathcal{S}^{\mp\vee}\otimes\mathcal{E}(-1))=
h^{q+1}(\mathcal{S}^{\pm\vee}\otimes\mathcal{E}(-1))$
for any $q$.
Thus $h^q(\mathcal{S}^{\mp\vee}\otimes\mathcal{E}(-1))=0$ for any $q$.

We apply to $\mathcal{E}(-1)$ 
the Bondal spectral sequence~\eqref{BondalSpectral}.
We see that $\Hom(G,\mathcal{E}(-1))=S_0^{\oplus 2}$,
that $\Ext^q(G,\mathcal{E}(-1))=0$ for $1\leq q\leq 3$,
and that $\Ext^4(G,\mathcal{E}(-1))=S_5^{\oplus r-2}$.
Hence $E_2^{p,q}=0$ unless $q=0$ or $q=4$,
$E_2^{p,0}=0$ unless $p=0$, $E_2^{0,0}=\mathcal{O}^{\oplus 2}$,
$E_2^{p,4}=0$ unless $p=-4$, and $E_2^{-4,4}=\mathcal{O}(-1)^{\oplus r-2}$
by Lemma~\ref{S2Arilemma} (2).
Therefore $\mathcal{E}(-1)$ fits in an exact sequence
\[
0\to \mathcal{O}^{\oplus 2}\to \mathcal{E}(-1)\to \mathcal{O}(-1)^{\oplus r-2}
\to 0.
\]
Hence $\mathcal{E}\cong \mathcal{O}(1)^{\oplus 2}\oplus\mathcal{O}^{\oplus r-2}$.
This is Case (2) of Theorem~\ref{Chern2OnQ4}.

\subsection{The case $n\geq 5$}
Suppose that 
$\mathcal{E}|_{\mathbb{Q}^{n-1}}\cong \mathcal{O}(1)^{\oplus 2}\oplus \mathcal{O}^{\oplus r-2}$.
We will show that 
$\mathcal{E}$ is isomorphic to $\mathcal{O}(1)^{\oplus 2}\oplus \mathcal{O}^{\oplus r-2}$
by applying to $\mathcal{E}(-1)$ 
the Bondal spectral sequence~\eqref{BondalSpectral}.

First, $h^q(\mathcal{E}(-2))=0$ for any $q$
by \eqref{firstvanishinggeq5}.
Hence we see that $h^0(\mathcal{E}(-1))=2$.
We have $h^q(\mathcal{E}(-1))=0$ for $q>0$ by 
\eqref{firstvanishing}.
Note that $h^q(\mathcal{E}(-t)|_{\mathbb{Q}^{n-1}})=0$ for any $q$ and any $t$
$(2\leq t\leq n-2)$.
Since $h^q(\mathcal{E}(-2))=0$ for any $q$,
this implies that $h^q(\mathcal{E}(-t))=0$ for any $q$ and any $t$
$(2\leq t\leq n-1)$.
Since $h^{n-1}(\mathcal{E}(-n+1)|_{\mathbb{Q}^{n-1}})=r-2$
and $h^q(\mathcal{E}(-n+1)|_{\mathbb{Q}^{n-1}})=0$ unless $q=n-1$,
we infer that $h^{n}(\mathcal{E}(-n))=r-2$
and that $h^q(\mathcal{E}(-n))=0$ unless $q=n$.

Note that  
$h^q( (\mathcal{S}^{(\pm)\vee}\otimes \mathcal{E}(-1))|_{\mathbb{Q}^{n-1}})=0$
for any $q$. Therefore 
the exact sequence~\eqref{RestrictionForSE(-1)} implies that 
$h^{q+1}(\mathcal{S}^{(\pm)\vee}\otimes \mathcal{E}(-2))
=h^{q+1}(\mathcal{S}^{(\pm)\vee}\otimes \mathcal{E}(-1))$ for any $q$.
This implies that 
$h^{q}(\mathcal{S}^{(\mp)\vee}\otimes \mathcal{E}(-1))=
h^{q+1}(\mathcal{S}^{(\pm)\vee}\otimes \mathcal{E}(-1))$ for any $q$
by \eqref{qq+1joushouForE(-2)n=5}.
Hence 
$\Ext^q(\mathcal{S}^{(\mp)},\mathcal{E}(-1))=0$ for  any $q$.

Now Lemma~\ref{S2Arilemma}
shows that $E_2^{p,q}=0$ unless 
$(p,q)=(0,0)$ or $(-n,n)$,
that $E_2^{0,0}=\mathcal{O}^{\oplus 2}$,
and that $E_2^{-n,n}=\mathcal{O}(-1)^{\oplus r-2}$.
Therefore $\mathcal{E}(-1)$ fits in an exact sequence
\[
0\to \mathcal{O}^{\oplus 2}\to \mathcal{E}(-1)\to \mathcal{O}(-1)^{\oplus r-2}
\to 0.
\]
Hence $\mathcal{E}\cong \mathcal{O}(1)^{\oplus 2}\oplus\mathcal{O}^{\oplus r-2}$.
This is Case (2) of Theorem~\ref{Chern2OnQ4}.

\section{The case where $\mathcal{E}|_{\mathbb{Q}^2}$ belongs to Case
(3) of Theorem~\ref{Chern2OnQ3}}

\subsection{The case $n=4$}
Suppose that $\mathcal{E}|_{\mathbb{Q}^3}
\cong 
\mathcal{O}(1)\oplus
\mathcal{S}\oplus    
\mathcal{O}^{\oplus r-3}$.
Then $c_2h^2=3$ and $c_3h=1$.
We will show that $\mathcal{E}
\cong
\mathcal{O}(1)\oplus
\mathcal{S}^{\pm}\oplus    
\mathcal{O}^{\oplus r-3}$
by applying
to $\mathcal{E}(-1)$
the Bondal spectral sequence~\eqref{BondalSpectral}.

We have $h^q(\mathcal{E}(-1))=0$ for $q>0$ by \eqref{firstvanishing}.
Note that $h^q(\mathcal{E}(-1)|_{\mathbb{Q}^3})=0$ for $q>0$
and that $h^0(\mathcal{E}(-1)|_{\mathbb{Q}^3})=1$.
Hence $h^q(\mathcal{E}(-2))=0$ for $q\geq 2$.
The assumption~\eqref{h^0(E(-2))vanish} together with \eqref{e(-2)RR}
and 
\eqref{c_2squarebiggerC1C3biggerC4}
shows that 
\[-h^1(\mathcal{E}(-2))=\chi(\mathcal{E}(-2))
=-\frac{5}{12}+\frac{1}{12}(c_2^2-2c_4)\geq -\dfrac{5+c_4}{12}.\]
Note here that $c_1c_3=2$ since $c_3h=1$. Hence $c_4\leq 2$ by \eqref{c_2squarebiggerC1C3biggerC4}.
Therefore $h^1(\mathcal{E}(-2))\leq
7/12$, and thus $h^1(\mathcal{E}(-2))=0$.
Since $h^0(\mathcal{E}(-1)|_{\mathbb{Q}^3})=1$,
this implies that $h^0(\mathcal{E}(-1))=1$.
Now that $h^q(\mathcal{E}(-2))=0$
for any $q$, we have 
$h^q(\mathcal{E}(-3))=h^{q-1}(\mathcal{E}(-2)|_{\mathbb{Q}^3})=0$
for any $q$.
Hence $h^q(\mathcal{E}(-4))=h^{q-1}(\mathcal{E}(-3)|_{\mathbb{Q}^3})$
for any $q$.
Note that $h^q(\mathcal{E}(-3)|_{\mathbb{Q}^3})=0$ for $q<3$
and that $h^3(\mathcal{E}(-3)|_{\mathbb{Q}^3})=r-3$.
Thus $h^q(\mathcal{E}(-4))=0$ for $q<4$
and $h^4(\mathcal{E}(-4))=r-3$.

Since $\mathcal{S}^{\vee}\otimes\mathcal{E}(-1)|_{\mathbb{Q}^3}\cong \mathcal{S}^{\vee}
\oplus (\mathcal{S}^{\vee}\otimes \mathcal{S}(-1))\oplus
\mathcal{S}^{\vee}(-1)^{\oplus r-3}$,
it follows from \cite[Lemma 8.2 (1)]{MR4453350}
that $h^1(\mathcal{S}^{\vee}\otimes\mathcal{E}(-1)|_{\mathbb{Q}^3})=1$
and that $h^q(\mathcal{S}^{\vee}\otimes\mathcal{E}(-1)|_{\mathbb{Q}^3})=0$
unless $q=1$.
Since $h^q(\mathcal{E}(-2))=0$ 
for any $q$,
the exact sequence~\eqref{E(-2)twistofSSdual} implies 
that
$h^q(\mathcal{S}^{\mp\vee}\otimes \mathcal{E}(-1))
=h^{q+1}(\mathcal{S}^{\pm\vee}\otimes \mathcal{E}(-2))$ for 
any $q$.
In particular, we have $h^4(\mathcal{S}^{\mp\vee}\otimes \mathcal{E}(-1))=0$
and $h^0(\mathcal{S}^{\pm\vee}\otimes \mathcal{E}(-2))=0$.
Since we have the exact sequence~\eqref{RestrictionForSE(-1)}
and $h^q(\mathcal{S}^{\vee}\otimes\mathcal{E}(-1)|_{\mathbb{Q}^3})=0$
unless $q=1$,
we infer that 
$h^0(\mathcal{S}^{\pm\vee}\otimes\mathcal{E}(-2))=
h^0(\mathcal{S}^{\pm\vee}\otimes\mathcal{E}(-1))$
and that 
$h^{q+1}(\mathcal{S}^{\pm\vee}\otimes\mathcal{E}(-2))=
h^{q+1}(\mathcal{S}^{\pm\vee}\otimes\mathcal{E}(-1))$ for $q\geq 2$.
Hence we have $h^0(\mathcal{S}^{\pm\vee}\otimes\mathcal{E}(-1))=0$
and $h^q(\mathcal{S}^{\mp\vee}\otimes\mathcal{E}(-1))
=h^{q+1}(\mathcal{S}^{\pm\vee}\otimes\mathcal{E}(-1))$ for $q\geq 2$. Consequently 
we see that $h^q(\mathcal{S}^{\mp\vee}\otimes\mathcal{E}(-1))=0$ for $q\geq 2$.
Note also that $h^1(\mathcal{S}^{\pm\vee}\otimes\mathcal{E}(-2))=0$
and that $h^q(\mathcal{S}^{\pm\vee}\otimes\mathcal{E}(-2))=0$ for $q\geq 3$.
Set $a=h^1(\mathcal{S}^{\pm}\otimes \mathcal{E}(-1))$.
Since $h^1(\mathcal{S}^{\vee}\otimes\mathcal{E}(-1)|_{\mathbb{Q}^3})=1$,
we see that $a=0$ or $1$ and that 
$1-a=h^2(\mathcal{S}^{\pm\vee}\otimes \mathcal{E}(-2))
=h^1(\mathcal{S}^{\mp\vee}\otimes \mathcal{E}(-1))$.
Therefore $(h^1(\mathcal{S}^{+\vee}\otimes \mathcal{E}(-1)),
h^1(\mathcal{S}^{-\vee}\otimes \mathcal{E}(-1)))$ is either $(0,1)$ or $(1,0)$.

We see that $\Ext^4(G,\mathcal{E}(-1))=S_5^{\oplus r-3}$,
that $\Ext^q(G,\mathcal{E}(-1))=0$ for $q=3,2$,
and that $\Hom(G,\mathcal{E}(-1))=S_0$.
Lemma~\ref{S2Arilemma} then shows 
that $E_2^{p,4}=0$ unless $p=-4$,
that $E_2^{-4,4}=\mathcal{O}(-1)^{\oplus r-3}$,
that $E_2^{p,q}=0$ for any $p$ if $q=3,2$, 
that $E_2^{p,0}=0$ unless $p=0$, 
and that $E_2^{0,0}=\mathcal{O}$.

If $(h^1(\mathcal{S}^{+\vee}\otimes \mathcal{E}(-1)),
h^1(\mathcal{S}^{-\vee}\otimes \mathcal{E}(-1)))=(0,1)$.
then $\Ext^1(G,\mathcal{E}(-1))=S_2$,
and Lemma~\ref{S2Arilemma} shows that 
that $E_2^{p,1}=0$ unless $p=-1$, and that $E_2^{-1,1}=\mathcal{S}^+(-1)$.

If $(h^1(\mathcal{S}^{+\vee}\otimes \mathcal{E}(-1)),
h^1(\mathcal{S}^{-\vee}\otimes \mathcal{E}(-1)))=(1,0)$.
then $\Ext^1(G,\mathcal{E}(-1))=S_1$,
and Lemma~\ref{S2Arilemma} shows that 
that $E_2^{p,1}=0$ unless $p=-1$, and that $E_2^{-1,1}=\mathcal{S}^-(-1)$.

Hence $\mathcal{E}(-1)$ has a filtration $\mathcal{O}\subset F^1(\mathcal{E}(-1))\subset \mathcal{E}(-1)$
such that $F^1(\mathcal{E}(-1))$ fits in the following exact sequences:
\begin{align*}
0\to \mathcal{O}\to F^1(\mathcal{E}(-1))&\to \mathcal{S}^{\pm}(-1)\to 0;\\
0\to F^1(\mathcal{E}(-1))\to \mathcal{E}(-1)&\to \mathcal{O}(-1)^{\oplus r-3}\to 0.
\end{align*}
Therefore $F^1(\mathcal{E}(-1))\cong  \mathcal{O}\oplus \mathcal{S}^{\pm}(-1)$,
and thus $\mathcal{E}\cong \mathcal{O}(1)\oplus \mathcal{S}^{\pm}\oplus \mathcal{O}^{\oplus r-3}$.
This is Case (3) of Theorem~\ref{Chern2OnQ4}.

\subsection{The case $n\geq 5$}
We will show that this case does not arise. 
Suppose, to the contrary, that there exists a nef vector bundle $\mathcal{E}$
on $\mathbb{Q}^5$ such that
$\mathcal{E}|_{\mathbb{Q}^4}\cong 
\mathcal{O}(1)\oplus \mathcal{S}^{\pm}\oplus \mathcal{O}^{\oplus r-3}$.
We have  
$h^q(\mathcal{S}^{\vee}\otimes \mathcal{E}(-1))
=h^{q+1}(\mathcal{S}^{\vee}\otimes \mathcal{E}(-2))$ for any $q$
by \eqref{qq+1joushouForE(-2)n=5}.
In particular, we have 
$h^0(\mathcal{S}^{\vee}\otimes \mathcal{E}(-2))=0$.
It follows from \cite[Lemma 8.2 (2)]{MR4453350}
that 
$h^q((\mathcal{S}^{+\vee}\oplus\mathcal{S}^{-\vee})\otimes\mathcal{E}(-1)|_{\mathbb{Q}^4})=0$
unless $q=1$
and that 
$h^1((\mathcal{S}^{+\vee}\oplus\mathcal{S}^{-\vee})\otimes\mathcal{E}(-1)|_{\mathbb{Q}^4})=1$.
Hence the exact sequence~\eqref{RestrictionForSE(-1)}
implies that 
$h^0(\mathcal{S}^{\vee}\otimes\mathcal{E}(-2))=
h^0(\mathcal{S}^{\vee}\otimes\mathcal{E}(-1))$
and that 
$h^{q+1}(\mathcal{S}^{\vee}\otimes\mathcal{E}(-2))=
h^{q+1}(\mathcal{S}^{\vee}\otimes\mathcal{E}(-1))$ for $q\geq 2$.
Hence we have $h^0(\mathcal{S}^{\vee}\otimes\mathcal{E}(-1))=0$
and $h^q(\mathcal{S}^{\vee}\otimes\mathcal{E}(-1))=0$ for $q\geq 2$.
Note also that $h^1(\mathcal{S}^{\vee}\otimes\mathcal{E}(-2))=0$.
Hence we have the following exact sequence:
\[
0\to H^1(\mathcal{S}^{\vee}\otimes\mathcal{E}(-1))
\to 
H^1((\mathcal{S}^{+\vee}\oplus\mathcal{S}^{-\vee})\otimes \mathcal{E}(-1)|_{\mathbb{Q}^4})
\to H^2(\mathcal{S}^{\vee}\otimes\mathcal{E}(-2))\to 0.\]
This is, however,  a contradiction, because
$h^1(\mathcal{S}^{\vee}\otimes \mathcal{E}(-1))
=h^{2}(\mathcal{S}^{\vee}\otimes \mathcal{E}(-2))$ and  
$h^1((\mathcal{S}^{+\vee}\oplus\mathcal{S}^{-\vee})\otimes \mathcal{E}(-1)|_{\mathbb{Q}^4})=1$.
Therefore the case $n\geq 5$ does not arise.

\section{The case where $\mathcal{E}|_{\mathbb{Q}^3}$ belongs to Case
(4) of Theorem~\ref{Chern2OnQ3}}

\subsection{The case $n\geq 4$}
Suppose that $\mathcal{E}|_{\mathbb{Q}^{n-1}}$ fits in an exact sequence
\[
0\to \mathcal{O}(-1)
\to \mathcal{O}(1)\oplus 
\mathcal{O}^{\oplus r}\to \mathcal{E}|_{\mathbb{Q}^{n-1}}\to 0.
\]
Then $h^0(\mathcal{E}(-1)|_{\mathbb{Q}^{n-1}})=1$
and $h^q(\mathcal{E}(-1)|_{\mathbb{Q}^{n-1}})=0$ for $q>0$.

We claim here that $h^q(\mathcal{E}(-2))=0$ for any $q$.
If $n\geq 5$, this follows from 
\eqref{firstvanishinggeq5}.
Suppose that $n=4$. 
Then $c_2h^2=4$ and  $c_3h=4$. 
Note that $h^q(\mathcal{E}(-1))=0$ for $q>0$ by \eqref{firstvanishing}.
Hence $h^q(\mathcal{E}(-2))=0$ for $q\geq 2$.
The assumption~\eqref{h^0(E(-2))vanish} together with \eqref{e(-2)RR}
and 
\eqref{c_2squarebiggerC1C3biggerC4}
shows that 
\[-h^1(\mathcal{E}(-2))=\chi(\mathcal{E}(-2))
=\frac{1}{12}(c_2^2-2c_4)\geq -\dfrac{c_4}{12}.\]
Since $c_3h=4$,
we have $c_1c_3=8$.  
Hence $c_4\leq 8$ by \eqref{c_2squarebiggerC1C3biggerC4}.
Therefore $h^1(\mathcal{E}(-2))\leq 2/3$, and thus $h^1(\mathcal{E}(-2))=0$.
Hence the claim holds.

Since $h^0(\mathcal{E}(-1)|_{\mathbb{Q}^{n-1}})=1$,
the claim above implies that $h^0(\mathcal{E}(-1))=1$.
Hence we have an injection $\mathcal{O}(1)\to \mathcal{E}$.
Let $\mathcal{F}$ be its cokernel. Then $\mathcal{F}$ is torsion-free
by \eqref{h^0(E(-2))vanish},
and 
we have the following exact sequence:
\[
0\to \mathcal{O}(1)\to \mathcal{E}\to \mathcal{F}\to 0.
\]
Since $\mathcal{O}_{\mathbb{Q}^{n-1}}(1)\to \mathcal{E}|_{\mathbb{Q}^{n-1}}$ is injective,
$\mathcal{F}|_{\mathbb{Q}^{n-1}}$ fits in the following exact sequences:
\[
0\to 
\mathcal{F}(-1)\to \mathcal{F}\to \mathcal{F}|_{\mathbb{Q}^{n-1}}\to 0;\]
\[
0\to \mathcal{O}(-1)
\to 
\mathcal{O}^{\oplus r}\to \mathcal{F}|_{\mathbb{Q}^{n-1}}\to 0.\]
We will apply to $\mathcal{F}$ 
the Bondal spectral sequence~\eqref{BondalSpectral}.
We have $h^q(\mathcal{F})=0$ for $q>0$,
since $h^q(\mathcal{E})=0$ for $q>0$ by \eqref{firstvanishing}.
Moreover $h^q(\mathcal{F}(-1))=0$ for any $q$,
since $h^q(\mathcal{E}(-1))=0$ for $q>0$ by \eqref{firstvanishing}. 
Thus $h^0(\mathcal{F})=h^0(\mathcal{F}|_{\mathbb{Q}^{n-1}})=r$.
Note that 
$h^q(\mathcal{F}(-k)|_{\mathbb{Q}^{n-1}})=0$ for any $q$
and any $k$ $(1\leq k\leq n-3)$,
that $h^{n-2}(\mathcal{F}(-n+2)|_{\mathbb{Q}^{n-1}})=1$,
and that $h^q(\mathcal{F}(-n+2))=0$ unless $q=n-2$.
Hence we see that
$h^q(\mathcal{F}(-k))=0$ for any $q$
and any $k$ $(2\leq k\leq n-2)$,
that $h^{n-1}(\mathcal{F}(-n+1))=1$,
and that $h^q(\mathcal{F}(-n+1))=0$ unless $q=n-1$. 
Since $\mathcal{F}$ is torsion-free, 
it follows from 
\eqref{SSdual}
that there is an 
exact sequence
\[
0\to \mathcal{S}^{(\pm)\vee}\otimes \mathcal{F}(-1)
\to 
\mathcal{F}(-1)^{\oplus \alpha}
\to 
\mathcal{S}^{(\mp)\vee}\otimes \mathcal{F}
\to 0,
\]
where $\alpha=2^{\lceil n/2\rceil}$.
Hence 
$h^q(\mathcal{S}^{(\mp)\vee}\otimes \mathcal{F})
=h^{q+1}(\mathcal{S}^{(\pm)\vee}\otimes \mathcal{F}(-1))$ for any $q$.
Note that 
$h^q(\mathcal{S}^{(\pm)\vee}\otimes \mathcal{F}|_{\mathbb{Q}^{n-1}})=0$
for any $q$.
Since we have an exact sequence
\[
0\to \mathcal{S}^{(\pm)\vee}\otimes \mathcal{F}(-1)
\to 
\mathcal{S}^{(\pm)\vee}\otimes \mathcal{F}
\to 
(\mathcal{S}^{(\pm)\vee}\otimes \mathcal{F})|_{\mathbb{Q}^3}
\to 0,
\]
we infer that 
$h^{q+1}(\mathcal{S}^{(\pm)\vee}\otimes\mathcal{F}(-1))=
h^{q+1}(\mathcal{S}^{(\pm)\vee}\otimes\mathcal{F})$ for any $q$.
Hence we see that
$h^q(\mathcal{S}^{(\pm)\vee}\otimes \mathcal{F})=0$ for any $q$.
Therefore $\Ext^q(G,\mathcal{F})=0$ unless $q=n-1$ or  $0$,
$\Ext^{n-1}(G,\mathcal{F})=S_n$ if $n$ is odd,
$\Ext^{n-1}(G,\mathcal{F})=S_{n+1}$ if $n$ is even,
and $\Hom(G,\mathcal{F})=S_0^{\oplus r}$.
Hence $E_2^{p,q}=0$ unless $(p.q)=(-n,n-1)$ or $(0,0)$,
$E_2^{-n,n-1}=\mathcal{O}(-1)$, and $E_2^{0,0}=\mathcal{O}^{\oplus r}$
by Lemma~\ref{S2Arilemma}.
Thus we have an exact sequence
\[
0\to \mathcal{O}(-1)\to \mathcal{O}^{\oplus r}\to \mathcal{F}\to 0.
\]
Therefore $\mathcal{E}$ belongs to Case (4) of Theorem~\ref{Chern2OnQ4}.

\section{The case where $\mathcal{E}|_{\mathbb{Q}^3}$ belongs to Case
(5) of Theorem~\ref{Chern2OnQ3}}

\subsection{The case $n=4$} 
Suppose that $\mathcal{E}|_{\mathbb{Q}^3}$ fits in the following exact sequence:
\[0\to \mathcal{O}^{\oplus a}\to\mathcal{S}^{\oplus 2}\oplus\mathcal{O}^{\oplus r-4+a}
\to \mathcal{E}|_{\mathbb{Q}^3}\to 0,\] where $a=0$ or $1$,
and the composite of the injection 
$\mathcal{O}^{\oplus a}\to\mathcal{S}^{\oplus 2}\oplus\mathcal{O}^{\oplus r-4+a}$
and the projection $\mathcal{S}^{\oplus 2}\oplus\mathcal{O}^{\oplus r-4+a}
\to \mathcal{O}^{\oplus r-4+a}$ is zero.
We will apply to $\mathcal{E}(-1)$ 
the Bondal spectral sequence~\eqref{BondalSpectral}.
First note that $h^q(\mathcal{E}(-1)|_{\mathbb{Q}^3})=0$
for any $q$ and that $h^q(\mathcal{E}(-1))=0$ for $q>0$
by \eqref{firstvanishing}.
Hence we see that $h^q(\mathcal{E}(-2))=0$ for 
any $q$ by  \eqref{h^0(E(-2))vanish}.
Thus we infer that $h^0(\mathcal{E}(-1))=0$.
Since $h^q(\mathcal{E}(-2)|_{\mathbb{Q}^3})=0$ for any $q$,
we infer that $h^q(\mathcal{E}(-3))=0$ for any $q$.
Note here that $h^q(\mathcal{E}(-3)|_{\mathbb{Q}^3})=0$
unless $q=2$ or $3$,
and that $h^2(\mathcal{E}(-3)|_{\mathbb{Q}^3})\leq a\leq 1$.
Set $b=h^2(\mathcal{E}(-3)|_{\mathbb{Q}^3})$. Then we see that 
$h^3(\mathcal{E}(-3)|_{\mathbb{Q}^3})=r-4+b$.
Moreover $h^3(\mathcal{E}(-3))=b$ and $h^4(\mathcal{E}(-4))=r-4+b$.

The exact sequence above induces the following exact sequence:
\[
0\to \mathcal{S}^{\vee}(-1)^{\oplus a}\to 
(\mathcal{S}^{\vee}\otimes \mathcal{S}(-1))^{\oplus 2}\oplus
\mathcal{S}^{\vee}(-1)^{\oplus r-4+a}
\to\mathcal{S}^{\vee}\otimes \mathcal{E}(-1)|_{\mathbb{Q}^3}\to 0.
\]
By \cite[Lemma 8.2 (1)]{MR4453350}, we see that 
$h^1(\mathcal{S}^{\vee}\otimes \mathcal{E}(-1)|_{\mathbb{Q}^3})=2$
and $h^q(\mathcal{S}^{\vee}\otimes \mathcal{E}(-1)|_{\mathbb{Q}^3})=0$
unless $q=1$.
Since $h^q(\mathcal{E}(-2))=0$ for any $q$, 
the exact sequence~\eqref{E(-2)twistofSSdual} implies  
that
$h^q(\mathcal{S}^{\mp\vee}\otimes \mathcal{E}(-1))
=h^{q+1}(\mathcal{S}^{\pm\vee}\otimes \mathcal{E}(-2))$ for any $q$.
Since $h^q(\mathcal{S}^{\vee}\otimes\mathcal{E}(-1)|_{\mathbb{Q}^3})=0$
unless $q=1$,
the exact sequence~\eqref{RestrictionForSE(-1)}
implies that 
$h^0(\mathcal{S}^{\pm\vee}\otimes\mathcal{E}(-2))=
h^0(\mathcal{S}^{\pm\vee}\otimes\mathcal{E}(-1))$
and that 
$h^{q+1}(\mathcal{S}^{\pm\vee}\otimes\mathcal{E}(-2))=
h^{q+1}(\mathcal{S}^{\pm\vee}\otimes\mathcal{E}(-1))$ for $q\geq 2$.
Hence we have $h^0(\mathcal{S}^{\pm\vee}\otimes\mathcal{E}(-1))=0$
and $h^q(\mathcal{S}^{\mp\vee}\otimes\mathcal{E}(-1))
=0$ for $q\geq 2$.
Note also that 
$h^1(\mathcal{S}^{\pm\vee}\otimes\mathcal{E}(-2))=0$.
Set $c=h^1(\mathcal{S}^{\pm\vee}\otimes \mathcal{E}(-1))$.
Since $h^1(\mathcal{S}^{\vee}\otimes\mathcal{E}(-1)|_{\mathbb{Q}^3})=2$,
we see that $0\leq c\leq 2$ and that 
$2-c=h^2(\mathcal{S}^{\pm\vee}\otimes \mathcal{E}(-2))
=h^1(\mathcal{S}^{\mp\vee}\otimes \mathcal{E}(-1))$.
Therefore $(h^1(\mathcal{S}^{+\vee}\otimes \mathcal{E}(-1)),
h^1(\mathcal{S}^{-\vee}\otimes \mathcal{E}(-1)))=(0,2)$, $(1,1)$, or $(2,0)$.
We may assume that $(h^1(\mathcal{S}^{+\vee}\otimes \mathcal{E}(-1)),
h^1(\mathcal{S}^{-\vee}\otimes \mathcal{E}(-1)))=(c,2-c)$.

Now we have the following equalities:
\begin{align*}
\Ext^4(G,\mathcal{E}(-1))&=S_5^{\oplus r-4+b};&
\Ext^3(G,\mathcal{E}(-1))&=S_5^{\oplus b};\\
\Ext^1(G,\mathcal{E}(-1))&=S_1^{\oplus c}\oplus S_2^{\oplus 2-c};&
\Ext^q(G,\mathcal{E}(-1))&=0\textrm{ for }q=2,0.
\end{align*}
Lemma~\ref{S2Arilemma} then shows 
that $E_2^{-4,4}=\mathcal{O}(-1)^{\oplus r-4+b}$,
that $E_2^{-4,3}=\mathcal{O}(-1)^{\oplus b}$,
that $E_2^{-1,1}=\mathcal{S}^+(-1)^{\oplus 2-c}\oplus\mathcal{S}^-(-1)^{\oplus c}$,
and that $E_2^{p,q}=0$ unless $(p,q)=(-4,4)$, $(-4,3)$, or $(-1,1)$. 
Hence we infer that $\mathcal{E}(-1)$ fits in the following exact sequence:
\[0\to \mathcal{O}(-1)^{\oplus b}
\to \mathcal{S}^+(-1)^{\oplus 2-c}\oplus\mathcal{S}^-(-1)^{\oplus c}
\to \mathcal{E}(-1)\to \mathcal{O}(-1)^{\oplus r-4+b}\to 0.\]
This yields Case (5) of Theorem~\ref{Chern2OnQ4}. 
More precisely, if $b=0$, then $\mathcal{E}$ is isomorphic to 
$\mathcal{S}^{+\oplus 2-c}\oplus\mathcal{S}^{-\oplus c}
\oplus  \mathcal{O}^{\oplus r-4}$.
If $b=1$, note that, for non-zero elements $s_1, s_2$ of $H^0(\mathcal{S}^{\pm})$,
the intersection $(s_1)_0\cap (s_2)_0$ can be empty if and only if 
$(s_1,s_2) \textrm{ or }(s_2,s_1)\in H^0(\mathcal{S}^{+})\times H^0(\mathcal{S}^{-})$.
Since $\mathcal{E}$ is a vector bundle, this implies that if $b=1$ then $c$ must be equal to $1$.

\subsection{The case $n=5$} 
Suppose that $\mathcal{E}|_{\mathbb{Q}^4}$ fits in the following exact sequence:
\[
0\to \mathcal{O}\to \mathcal{S}^{\pm}\oplus\mathcal{S}^{\pm}\oplus\mathcal{O}^{\oplus r-3}
\to\mathcal{E}|_{\mathbb{Q}^4}\to 0.
\]
We will apply to $\mathcal{E}(-1)$ 
the Bondal spectral sequence~\eqref{BondalSpectral}.
First note that $h^q(\mathcal{E}(-1)|_{\mathbb{Q}^4})=0$
for any $q$ and that $h^q(\mathcal{E}(-2))=0$ for any $q$
by 
\eqref{firstvanishinggeq5}.
Hence we see that $h^q(\mathcal{E}(-1))=0$ for any $q$.
Since $h^q(\mathcal{E}(-2)|_{\mathbb{Q}^4})=0$ for any $q$,
we infer that $h^q(\mathcal{E}(-3))=0$ for any $q$.
Furthermore $h^q(\mathcal{E}(-4))=0$ for any $q$,
since $h^q(\mathcal{E}(-3)|_{\mathbb{Q}^4})=0$ for any $q$.
Note here that $h^q(\mathcal{E}(-4)|_{\mathbb{Q}^4})=0$
unless $q=3$ or $4$,
and that $h^3(\mathcal{E}(-4)|_{\mathbb{Q}^4})\leq 1$.
Set $a=h^3(\mathcal{E}(-4)|_{\mathbb{Q}^4})$. Then we see that 
$h^4(\mathcal{E}(-4)|_{\mathbb{Q}^4})=r-4+a$.
Moreover $h^4(\mathcal{E}(-5))=a$ and $h^5(\mathcal{E}(-5))=r-4+a$.

The exact sequence above induces the following exact sequence:
\[
0\to \mathcal{S}^{\vee}(-1)|_{\mathbb{Q}^4}\to 
(\mathcal{S}^{+\vee}\oplus \mathcal{S}^{-\vee})
\otimes
(\mathcal{S}^{\pm}\otimes \mathcal{S}^{\pm}\oplus
\mathcal{O}^{\oplus r-3})(-1)
\to(\mathcal{S}^{\vee}\otimes \mathcal{E}(-1))|_{\mathbb{Q}^4}\to 0.
\]
By \cite[Lemma 8.2 (2)]{MR4453350}, we see that 
$h^1((\mathcal{S}^{\vee}\otimes \mathcal{E}(-1))|_{\mathbb{Q}^4})=2$
and that $h^q((\mathcal{S}^{\vee}\otimes \mathcal{E}(-1))|_{\mathbb{Q}^4})=0$
unless $q=1$.
By \eqref{qq+1joushouForE(-2)n=5},
we have 
$h^q(\mathcal{S}^{\vee}\otimes \mathcal{E}(-1))
=h^{q+1}(\mathcal{S}^{\vee}\otimes \mathcal{E}(-2))$ for any $q$.
Since $h^q((\mathcal{S}^{\vee}\otimes\mathcal{E}(-1))|_{\mathbb{Q}^4})=0$
unless $q=1$,
the exact sequence~\eqref{RestrictionForSE(-1)} implies 
that 
$h^0(\mathcal{S}^{\vee}\otimes\mathcal{E}(-2))=
h^0(\mathcal{S}^{\vee}\otimes\mathcal{E}(-1))$
and that 
$h^{q+1}(\mathcal{S}^{\vee}\otimes\mathcal{E}(-2))=
h^{q+1}(\mathcal{S}^{\vee}\otimes\mathcal{E}(-1))$ for $q\geq 2$.
Hence we have $h^0(\mathcal{S}^{\vee}\otimes\mathcal{E}(-1))=0$
and $h^q(\mathcal{S}^{\vee}\otimes\mathcal{E}(-1))
=0$ for $q\geq 2$.
Note also that $h^1(\mathcal{S}^{\vee}\otimes\mathcal{E}(-2))=0$.
Since $h^1((\mathcal{S}^{\vee}\otimes\mathcal{E}(-1))|_{\mathbb{Q}^4})=2$
and $h^1(\mathcal{S}^{\vee}\otimes \mathcal{E}(-1))
=h^2(\mathcal{S}^{\vee}\otimes \mathcal{E}(-2))$,
we see that $h^1(\mathcal{S}^{\vee}\otimes \mathcal{E}(-1))=1$.

Now we have the following equalities:
\begin{align*}
\Ext^5(G,\mathcal{E}(-1))&=S_5^{\oplus r-4+a};&
\Ext^4(G,\mathcal{E}(-1))&=S_5^{\oplus a};\\
\Ext^1(G,\mathcal{E}(-1))&=S_1;&
\Ext^q(G,\mathcal{E}(-1))&=0\textrm{ for }q=3,2,0.
\end{align*}
Lemma~\ref{S2Arilemma} then shows 
that $E_2^{-5,5}=\mathcal{O}(-1)^{\oplus r-4+a}$,
that $E_2^{-5,4}=\mathcal{O}(-1)^{\oplus a}$,
that $E_2^{-1,1}=\mathcal{S}(-1)$,
and that $E_2^{p,q}=0$ unless $(p,q)=(-5,5)$, $(-5,4)$, or $(-1,1)$. 
Hence we infer that $\mathcal{E}(-1)$ fits in the following exact sequence:
\[0\to \mathcal{O}(-1)^{\oplus a}
\to \mathcal{S}(-1)
\to \mathcal{E}(-1)\to \mathcal{O}(-1)^{\oplus r-4+a}\to 0.\]
This yields Case (6) of Theorem~\ref{Chern2OnQ4}. 
Note that $c_4(\mathcal{S})=0$ since $c_4(\mathcal{S})h=0$.

\subsection{The case $n=6$} 
Suppose that $\mathcal{E}|_{\mathbb{Q}^5}$
fits in the following exact sequence:
\[
0\to \mathcal{O}\to \mathcal{S}\oplus\mathcal{O}^{\oplus r-3}
\to\mathcal{E}|_{\mathbb{Q}^5}\to 0.
\]
We will apply to $\mathcal{E}(-1)$ 
the Bondal spectral sequence~\eqref{BondalSpectral}.
First note that $h^q(\mathcal{E}(-k)|_{\mathbb{Q}^5})=0$
for any $q$ if $1\leq k\leq 4$ and that $h^q(\mathcal{E}(-k))=0$ for any $q$
if $1\leq k\leq 2$
by \eqref{h^0(E(-2))vanish},
\eqref{firstvanishing},
and \eqref{firstvanishinggeq5}.
Hence we also see that $h^q(\mathcal{E}(-k))=0$ for any $q$
if $3\leq k\leq 5$.
Note here that $h^q(\mathcal{E}(-5)|_{\mathbb{Q}^5})=0$
unless $q=4$ or $5$,
and that $h^4(\mathcal{E}(-5)|_{\mathbb{Q}^5})\leq 1$.
Set $a=h^4(\mathcal{E}(-5)|_{\mathbb{Q}^5})$. Then we see that 
$h^5(\mathcal{E}(-5)|_{\mathbb{Q}^5})=r-4+a$.
Moreover $h^5(\mathcal{E}(-6))=a$ and $h^6(\mathcal{E}(-6))=r-4+a$.

The exact sequence above induces the following exact sequence:
\[
0\to \mathcal{S}^{\vee}(-1)\to 
\mathcal{S}^{\vee}
\otimes
\mathcal{S}(-1)\oplus
\mathcal{S}^{\vee}(-1)^{\oplus r-3}
\to\mathcal{S}^{\vee}\otimes \mathcal{E}(-1)|_{\mathbb{Q}^5}\to 0.
\]
By \cite[Lemma 8.2 (1)]{MR4453350}, we see that 
$h^1((\mathcal{S}^{\vee}\otimes \mathcal{E}(-1))|_{\mathbb{Q}^5})=1$
and that $h^q((\mathcal{S}^{\vee}\otimes \mathcal{E}(-1))|_{\mathbb{Q}^4})=0$
unless $q=1$.
By \eqref{qq+1joushouForE(-2)n=5}, we have 
$h^q(\mathcal{S}^{\mp\vee}\otimes \mathcal{E}(-1))
=h^{q+1}(\mathcal{S}^{\pm\vee}\otimes \mathcal{E}(-2))$ for all $q$.
Since $h^q(\mathcal{S}^{\vee}\otimes\mathcal{E}(-1)|_{\mathbb{Q}^5})=0$
unless $q=1$,
the exact sequence~\eqref{RestrictionForSE(-1)} implies that 
$h^0(\mathcal{S}^{\pm\vee}\otimes\mathcal{E}(-2))=
h^0(\mathcal{S}^{\pm\vee}\otimes\mathcal{E}(-1))$
and that 
$h^{q+1}(\mathcal{S}^{\pm\vee}\otimes\mathcal{E}(-2))=
h^{q+1}(\mathcal{S}^{\pm\vee}\otimes\mathcal{E}(-1))$ for $q\geq 2$.
Hence we have $h^0(\mathcal{S}^{\pm\vee}\otimes\mathcal{E}(-1))=0$
and $h^q(\mathcal{S}^{\mp\vee}\otimes\mathcal{E}(-1))
=0$ for $q\geq 2$.
Note also that $h^1(\mathcal{S}^{\mp\vee}\otimes\mathcal{E}(-2))=0$.
Set $b=h^1(\mathcal{S}^{+\vee}\otimes \mathcal{E}(-1))$.
Since $h^1(\mathcal{S}^{\vee}\otimes\mathcal{E}(-1)|_{\mathbb{Q}^5})=1$,
we see that  $b=0$ or $1$ and that 
$1-b=h^2(\mathcal{S}^{+\vee}\otimes \mathcal{E}(-2))
=h^1(\mathcal{S}^{-\vee}\otimes \mathcal{E}(-1))$.
Hence we see that 
$(h^1(\mathcal{S}^{+\vee}\otimes \mathcal{E}(-1)),
h^1(\mathcal{S}^{-\vee}\otimes \mathcal{E}(-1)))=(0,1)$, or $(1,0)$.

Now we have the following equalities:
\begin{align*}
\Ext^6(G,\mathcal{E}(-1))&=S_7^{\oplus r-4+a};&
\Ext^5(G,\mathcal{E}(-1))&=S_7^{\oplus a};\\
\Ext^1(G,\mathcal{E}(-1))&=S_2\textrm{ if }b=0;&
\Ext^1(G,\mathcal{E}(-1))&=S_1\textrm{ if }b=1;\\
\Ext^q(G,\mathcal{E}(-1))&=0\textrm{ unless }q=6,5,1.&
&
\end{align*}
Lemma~\ref{S2Arilemma} then shows 
that $E_2^{-6,6}=\mathcal{O}(-1)^{\oplus r-4+a}$,
that $E_2^{-6,5}=\mathcal{O}(-1)^{\oplus a}$,
that $E_2^{-1,1}=\mathcal{S}^{\pm}(-1)$,
and that $E_2^{p,q}=0$ unless $(p,q)=(-6,6)$, $(-6,5)$, or $(-1,1)$. 
Hence we infer that $\mathcal{E}(-1)$ fits in the following exact sequence:
\[0\to \mathcal{O}(-1)^{\oplus a}
\to \mathcal{S}^{\pm}(-1)
\to \mathcal{E}(-1)\to \mathcal{O}(-1)^{\oplus r-4+a}\to 0.\]
This yields Case (7) of Theorem~\ref{Chern2OnQ4}. 
Note that $c_4(\mathcal{S}^{\pm})=0$ since $c_4(\mathcal{S}^{\pm})h^2=0$.

\subsection{The case $n\geq 7$}
We will show that this case does not arise. 
Suppose, to the contrary, that there exists a nef vector bundle $\mathcal{E}$
on $\mathbb{Q}^7$ such that
$\mathcal{E}|_{\mathbb{Q}^6}$ fits in the following exact sequence:
\[
0\to \mathcal{O}\to \mathcal{S}^{\pm}\oplus\mathcal{O}^{\oplus r-3}
\to\mathcal{E}|_{\mathbb{Q}^6}\to 0.
\]
We have 
$h^q(\mathcal{S}^{\vee}\otimes \mathcal{E}(-1))
=h^{q+1}(\mathcal{S}^{\vee}\otimes \mathcal{E}(-2))$ for any $q$
by \eqref{qq+1joushouForE(-2)n=5}.
By \cite[Lemma 8.2 (2)]{MR4453350}, we see that 
$h^1((\mathcal{S}^{+\vee}\oplus\mathcal{S}^{-\vee})\otimes \mathcal{E}(-1)|_{\mathbb{Q}^4})=1$
and that 
$h^q((\mathcal{S}^{+\vee}\oplus\mathcal{S}^{-\vee})\otimes\mathcal{E}(-1)|_{\mathbb{Q}^3})=0$
unless $q=1$.
Hence  
the exact sequence~\eqref{RestrictionForSE(-1)} implies that 
$h^0(\mathcal{S}^{\vee}\otimes\mathcal{E}(-2))=
h^0(\mathcal{S}^{\vee}\otimes\mathcal{E}(-1))$
and that 
$h^{q+1}(\mathcal{S}^{\vee}\otimes\mathcal{E}(-2))=
h^{q+1}(\mathcal{S}^{\vee}\otimes\mathcal{E}(-1))$ for $q\geq 2$.
Hence we have $h^0(\mathcal{S}^{\vee}\otimes\mathcal{E}(-1))=0$
and $h^q(\mathcal{S}^{\vee}\otimes\mathcal{E}(-1))=0$ for $q\geq 2$.
Note also that $h^1(\mathcal{S}^{\vee}\otimes\mathcal{E}(-2))=0$.
Hence we have the following exact sequence:
\[
0\to H^1(\mathcal{S}^{\vee}\otimes\mathcal{E}(-1))
\to
H^1((\mathcal{S}^{+\vee}\oplus\mathcal{S}^{-\vee})\otimes \mathcal{E}(-1)|_{\mathbb{Q}^6})
\to H^2(\mathcal{S}^{\vee}\otimes\mathcal{E}(-2))\to 0.\]
This is, however,  a contradiction, because
$h^1(\mathcal{S}^{\vee}\otimes \mathcal{E}(-1))
=h^{2}(\mathcal{S}^{\vee}\otimes \mathcal{E}(-2))$ and  
$h^1((\mathcal{S}^{+\vee}\oplus\mathcal{S}^{-\vee})\otimes \mathcal{E}(-1)|_{\mathbb{Q}^4})=1$.
Therefore the case $n\geq 7$ does not arise.

\section{The case where $\mathcal{E}|_{\mathbb{Q}^3}$ belongs to Case
(6) of Theorem~\ref{Chern2OnQ3}}

\subsection{The case $n=4$} 
Suppose that $\mathcal{E}|_{\mathbb{Q}^3}$ fits in the following exact sequence:
\[
0\to \mathcal{S}(-1)\oplus\mathcal{O}(-1)\to \mathcal{O}^{\oplus r+3}
\to\mathcal{E}|_{\mathbb{Q}^3}\to 0.
\]
We will apply to $\mathcal{E}$ 
the Bondal spectral sequence~\eqref{BondalSpectral}.
First note that $h^q(\mathcal{E}(-1)|_{\mathbb{Q}^3})=0$
for all $q$.
Hence we have 
$h^q(\mathcal{E}(-2))=h^q(\mathcal{E}(-1))$ for any $q$.
Since we have \eqref{h^0(E(-2))vanish} and \eqref{firstvanishing},
this implies that $h^q(\mathcal{E}(-k))=0$ for $k=1,2$ and any $q$.
Thus we see that $h^0(\mathcal{E})=h^0(\mathcal{E}|_{\mathbb{Q}^3})=r+3$.
Note here that $h^2(\mathcal{E}(-2)|_{\mathbb{Q}^3})=1$
and that $h^q(\mathcal{E}(-2)|_{\mathbb{Q}^3})=0$ unless $q=2$.
Hence we see that $h^3(\mathcal{E}(-3))=1$
and that $h^q(\mathcal{E}(-3))=0$ unless $q=3$.

The exact sequence above induces the following exact sequence:
\[
0\to (\mathcal{S}^{\vee}\otimes \mathcal{S}(-1))\oplus\mathcal{S}^{\vee}(-1)
\to \mathcal{S}^{\vee\oplus r+3}
\to\mathcal{S}^{\vee}\otimes \mathcal{E}|_{\mathbb{Q}^3}\to 0.
\]
By \cite[Lemma 8.2 (1)]{MR4453350}, we see that 
$h^0(\mathcal{S}^{\vee}\otimes \mathcal{E}|_{\mathbb{Q}^3})=1$
and $h^q(\mathcal{S}^{\vee}\otimes \mathcal{E}|_{\mathbb{Q}^3})=0$
unless $q=1$.
We have $h^q(\mathcal{S}^{\mp\vee}\otimes \mathcal{E})
=h^{q+1}(\mathcal{S}^{\pm\vee}\otimes \mathcal{E}(-1))$ for any $q$
by \eqref{qq+1joushouForE(-1)}.
In particular, we have $h^0(\mathcal{S}^{\pm\vee}\otimes \mathcal{E}(-1))=0$.
Since $h^q(\mathcal{S}^{\vee}\otimes\mathcal{E}|_{\mathbb{Q}^3})=0$
unless $q=0$,
the exact sequence~\eqref{RestrictionForSE} implies that
$h^{q+1}(\mathcal{S}^{\pm\vee}\otimes\mathcal{E}(-1))=
h^{q+1}(\mathcal{S}^{\pm\vee}\otimes\mathcal{E})$ for $q\geq 1$.
Hence we have $h^q(\mathcal{S}^{\mp\vee}\otimes\mathcal{E})
=0$ for $q\geq 1$.
Set $a=h^0(\mathcal{S}^{+\vee}\otimes \mathcal{E})$.
Since $h^0(\mathcal{S}^{\vee}\otimes\mathcal{E}|_{\mathbb{Q}^3})=1$,
we see that $a=0$ or $1$ and that 
$1-a=h^1(\mathcal{S}^{+\vee}\otimes \mathcal{E}(-1))
=h^0(\mathcal{S}^{-\vee}\otimes \mathcal{E})$.
Therefore $(h^0(\mathcal{S}^{+\vee}\otimes \mathcal{E}),
h^0(\mathcal{S}^{-\vee}\otimes \mathcal{E}))=(0,1)$, or $(1,0)$.

Now we see that $\Ext^3(G,\mathcal{E})=S_5$, that $\Ext^q(G,\mathcal{E})=0$ for $q=2,1$,
and that $\Hom(G,\mathcal{E})$ fits in the following exact sequence:
\[0\to S_0^{\oplus r+3}\to \Hom(G,\mathcal{E})\to S_2\to 0 \quad \textrm{ if } a=0;\]
\[0\to S_0^{\oplus r+3}\to \Hom(G,\mathcal{E})\to S_1\to 0 \quad \textrm{ if } a=1.\]
Lemma~\ref{S2Arilemma} then shows 
that $E_2^{-4,3}=\mathcal{O}(-1)$,
that $E_2^{p,q}=0$ unless $(p,q)=(-4,3)$, or $(0,0)$,
and that $E_2^{0,0}$ fits in the following exact sequence:
\[
0\to \mathcal{S}^{\pm}(-1)\to \mathcal{O}^{\oplus r+3}\to E_2^{0,0}\to 0.
\]
Hence we infer that $\mathcal{E}$ fits in the following exact sequence:
\[0\to \mathcal{S}^{\pm}(-1)\oplus \mathcal{O}(-1)
\to \mathcal{O}^{\oplus r+3}
\to \mathcal{E}\to 0.\]
This is Case (8) of Theorem~\ref{Chern2OnQ4}. 

\subsection{The case $n\geq 5$}
We will show that this case does not arise. 
Suppose, to the contrary, that there exists a nef vector bundle $\mathcal{E}$
on $\mathbb{Q}^5$ such that
$\mathcal{E}|_{\mathbb{Q}^4}$ fits in the following exact sequence:
\[0\to \mathcal{S}^{\pm}(-1)\oplus \mathcal{O}(-1)
\to \mathcal{O}^{\oplus r+3}
\to \mathcal{E}|_{\mathbb{Q}^4}\to 0.\]
Since 
we have 
$h^0(\mathcal{E}(-1)|_{\mathbb{Q}^4})=0$,
it follows from \eqref{h^0(E(-2))vanish}
that $h^0(\mathcal{E}(-1))=0$.
Hence 
we have 
$h^q(\mathcal{S}^{\vee}\otimes \mathcal{E})
=h^{q+1}(\mathcal{S}^{\vee}\otimes \mathcal{E}(-1))$ for any $q$
by \eqref{qq+1joushouForE(-1)}.
In particular, 
$h^0(\mathcal{S}^{\vee}\otimes \mathcal{E}(-1))=0$.
Note that 
$h^q((\mathcal{S}^{+\vee}\oplus\mathcal{S}^{-\vee})\otimes\mathcal{E}|_{\mathbb{Q}^4})=0$
unless $q=0$,
and that 
$h^0((\mathcal{S}^{+\vee}\oplus\mathcal{S}^{-\vee})\otimes \mathcal{E}|_{\mathbb{Q}^4})=1$
by \cite[Lemma 8.2 (2)]{MR4453350}.
Hence the exact sequence~\eqref{RestrictionForSE} implies that 
$h^{q+1}(\mathcal{S}^{\vee}\otimes\mathcal{E}(-1))=
h^{q+1}(\mathcal{S}^{\vee}\otimes\mathcal{E})$ for $q\geq 1$.
Hence we have $h^q(\mathcal{S}^{\vee}\otimes\mathcal{E})=0$ for $q\geq 1$.
Now we have the following exact sequence:
\[
0\to H^0(\mathcal{S}^{\vee}\otimes\mathcal{E})
\to H^0((\mathcal{S}^{+\vee}\oplus\mathcal{S}^{-\vee})\otimes \mathcal{E}|_{\mathbb{Q}^4})
\to H^1(\mathcal{S}^{\vee}\otimes\mathcal{E}(-1))\to 0.\]
This is, however,  a contradiction, because
$h^0(\mathcal{S}^{\vee}\otimes \mathcal{E})
=h^{1}(\mathcal{S}^{\vee}\otimes \mathcal{E}(-1))$ and  
$h^0((\mathcal{S}^{+\vee}\oplus\mathcal{S}^{-\vee})\otimes \mathcal{E}|_{\mathbb{Q}^4})=1$.
Therefore the case $n\geq 5$ does not arise.

\section{The case where $\mathcal{E}|_{\mathbb{Q}^3}$ belongs to Case
(7) of Theorem~\ref{Chern2OnQ3}}

\subsection{The case $n\geq 4$} 
Suppose that $\mathcal{E}|_{\mathbb{Q}^{n-1}}$ fits in an exact sequence
\[
0\to \mathcal{O}(-1)^{\oplus 2}
\to \mathcal{O}^{\oplus r+2}\to \mathcal{E}|_{\mathbb{Q}^{n-1}}\to 0.
\]
Then we see that $h^0(\mathcal{E}|_{\mathbb{Q}^{n-1}})=r+2$,
that $h^q(\mathcal{E}|_{\mathbb{Q}^{n-1}})=0$ for $q>0$,
that $h^q(\mathcal{E}(-k)|_{\mathbb{Q}^{n-1}})=0$ for any $q$
if $1\leq k\leq n-3$,
that $h^{n-2}(\mathcal{E}(-n+2)|_{\mathbb{Q}^{n-1}})=2$,
and that $h^{q}(\mathcal{E}(-n+2)|_{\mathbb{Q}^{n-1}})=0$
unless $q=n-2$.

We will apply to $\mathcal{E}$ 
the Bondal spectral sequence~\eqref{BondalSpectral}.
Since we have \eqref{h^0(E(-2))vanish}, we see that $h^0(\mathcal{E}(-1))=0$.
Hence 
$h^q(\mathcal{E}(-1))=0$ for 
any $q$ by \eqref{firstvanishing}.
Then we see 
that $h^q(\mathcal{E}(-k))=0$ for any $q$
if $1\leq k\leq n-2$,
that $h^0(\mathcal{E})=r+2$,
that $h^q(\mathcal{E})=0$ for $q>0$,
that $h^{n-1}(\mathcal{E}(-n+1))=2$,
and that $h^{q}(\mathcal{E}(-n+1))=0$
unless $q=n-1$.

We have $h^q(\mathcal{S}^{(\mp)\vee}\otimes \mathcal{E})
=h^{q+1}(\mathcal{S}^{(\pm)\vee}\otimes \mathcal{E}(-1))$ for any $q$
by \eqref{qq+1joushouForE(-1)}.
Note that 
$h^q((\mathcal{S}^{(\pm)\vee}\otimes \mathcal{E})|_{\mathbb{Q}^{n-1}})=0$
for any $q$.
The exact sequence~\eqref{RestrictionForSE} then implies that
$h^{q+1}(\mathcal{S}^{(\pm)\vee}\otimes\mathcal{E}(-1))=
h^{q+1}(\mathcal{S}^{(\pm)\vee}\otimes\mathcal{E})$ for any $q$.
Hence we see that
$h^q(\mathcal{S}^{(\pm)\vee}\otimes \mathcal{E})=0$ for any $q$.

Thus we see that
$\Ext^q(G,\mathcal{E})=0$ unless $q=n-1$ or  $0$,
that $\Ext^{n-1}(G,\mathcal{E})=S_n^{\oplus 2}$ if $n$ is odd,
that $\Ext^{n-1}(G,\mathcal{E})=S_{n+1}^{\oplus 2}$ if $n$ is even,
and that $\Hom(G,\mathcal{E})=S_0^{\oplus r+2}$.
Hence $E_2^{p,q}=0$ unless $(p.q)=(-n,n-1)$ or $(0,0)$,
$E_2^{-n,n-1}=\mathcal{O}(-1)^{\oplus 2}$, and $E_2^{0,0}=\mathcal{O}^{\oplus r}$
by Lemma~\ref{S2Arilemma}.
Therefore we have 
the following exact sequence:
\[
0\to \mathcal{O}(-1)^{\oplus 2}\to \mathcal{O}^{\oplus r+2}\to \mathcal{E}\to 0.
\]
This is Case (9) of Theorem~\ref{Chern2OnQ4}.

\section{The case where $\mathcal{E}|_{\mathbb{Q}^3}$ belongs to Case
(8) of Theorem~\ref{Chern2OnQ3}}

\subsection{The case $n\geq 4$} 
Suppose that $\mathcal{E}|_{\mathbb{Q}^{n-1}}$ fits in an exact sequence
\[
0\to \mathcal{O}(-2)
\to \mathcal{O}^{\oplus r+1}\to \mathcal{E}|_{\mathbb{Q}^{n-1}}\to 0.
\]
Then we see that $h^0(\mathcal{E}|_{\mathbb{Q}^{n-1}})=r+1$,
that $h^q(\mathcal{E}|_{\mathbb{Q}^{n-1}})=0$ for $q>0$,
that $h^q(\mathcal{E}(-k)|_{\mathbb{Q}^{n-1}})=0$ for any $q$
if $1\leq k\leq n-4$,
that $h^{n-2}(\mathcal{E}(-n+3)|_{\mathbb{Q}^{n-1}})=1$,
that $h^{q}(\mathcal{E}(-n+3)|_{\mathbb{Q}^{n-1}})=0$
unless $q=n-2$,
that $h^{n-2}(\mathcal{E}(-n+2)|_{\mathbb{Q}^{n-1}})=n+1$,
and that $h^{q}(\mathcal{E}(-n+2)|_{\mathbb{Q}^{n-1}})=0$
unless $q=n-2$.

We will apply to $\mathcal{E}$ 
the Bondal spectral sequence~\eqref{BondalSpectral}.
We have $h^0(\mathcal{E}(-1))=0$ by \eqref{h^0(E(-2))vanish}.
Then it follows from \eqref{firstvanishing}
that $h^0(\mathcal{E})=r+1$,
that $h^q(\mathcal{E})=0$ for $q>0$,
that 
$h^q(\mathcal{E}(-k))=0$ for any $q$
if $1\leq k\leq n-3$,
that $h^{n-1}(\mathcal{E}(-n+2))=1$,
that $h^{q}(\mathcal{E}(-n+2))=0$
unless $q=n-1$,
that $h^{n-1}(\mathcal{E}(-n+1))=n+2$,
and that $h^{q}(\mathcal{E}(-n+1))=0$
unless $q=n-1$.

We have 
$h^q(\mathcal{S}^{(\mp)\vee}\otimes \mathcal{E})
=h^{q+1}(\mathcal{S}^{(\pm)\vee}\otimes \mathcal{E}(-1))$ for any $q$
by \eqref{qq+1joushouForE(-1)}.
Note that 
$h^q((\mathcal{S}^{(\pm)\vee}\otimes \mathcal{E})|_{\mathbb{Q}^{n-1}})=0$
for any $q$, since $n-1\geq 3$.
The exact sequence~\eqref{RestrictionForSE} then implies that 
$h^{q+1}(\mathcal{S}^{(\pm)\vee}\otimes\mathcal{E}(-1))=
h^{q+1}(\mathcal{S}^{(\pm)\vee}\otimes\mathcal{E})$ for any $q$.
Hence we see that
$h^q(\mathcal{S}^{(\pm)\vee}\otimes \mathcal{E})=0$ for any $q$.

Thus $\Ext^q(G,\mathcal{E})=0$ unless $q=n-1$ or  $0$,
$\Hom(G,\mathcal{E})=S_0^{\oplus r+1}$,
and $\Ext^{n-1}(G,\mathcal{E})$ fits in the following exact sequence:
\[0\to S_{n-1}\to \Ext^{n-1}(G,\mathcal{E})\to S_n^{\oplus n+2}\to 0 \quad \textrm{ if } n
\textrm{ is odd};\]
\[0\to S_n\to \Ext^{n-1}(G,\mathcal{E})\to S_{n+1}^{\oplus n+2}\to 0 \quad \textrm{ if } n
\textrm{ is even}.\]
Therefore Lemma~\ref{S2Arilemma} implies 
that $E_2^{p,q}=0$ unless $(p,q)=(-n,n-1)$ $(-n+1,n-1)$ or $(0,0)$,
that $E_2^{0,0}\cong \mathcal{O}^{\oplus r+1}$,
and 
that $E_2^{-n,n-1}$ and $E_2^{-n+1,n-1}$ fit in the following exact sequence:
\begin{equation}\label{(13)noE_2^22noMotonoExSeq}
0\to E_2^{-n,n-1}\to \mathcal{O}(-1)^{\oplus n+2}\to T_{\mathbb{P}^{n+1}}(-2)|_{\mathbb{Q}^n}
\to E_2^{-n+1,n-1}\to 0.
\end{equation}
The Bondal spectral sequence induces the following isomorphisms
and exact sequences:
\[E_2^{-n,n-1}\cong E_n^{-n,n-1};\]
\[E_2^{0,0}\cong E_n^{0,0};\]
\[
0\to E_n^{-n,n-1}\to E_n^{0,0}\to E_{n+1}^{0,0}\to 0;
\]
\[
0\to E_{n+1}^{0,0}\to \mathcal{E}\to E_2^{-n+1,n-1}\to 0.
\]
Note here that $E_2^{-n+1,n-1}|_L$ cannot admit a negative degree quotient 
for any line $L$ in 
$\mathbb{Q}^n$ since $\mathcal{E}$ is nef.
We will show that $E_2^{-n+1,n-1}=0$.
First note that the exact sequence \eqref{(13)noE_2^22noMotonoExSeq}
induces the following exact sequence:
\[
0\to E_2^{-n,n-1}\to \mathcal{O}(-1)^{\oplus n+2}
\oplus \mathcal{O}(-2)
\xrightarrow{p}
\mathcal{O}(-1)^{\oplus n+2}
\to E_2^{-n+1,n-1}\to 0.
\]
Consider the composite of the inclusion 
$\mathcal{O}(-1)^{\oplus n+2}\to  \mathcal{O}(-1)^{\oplus n+2}
\oplus \mathcal{O}(-2)$ and the morphism $p$ above,
and let $\mathcal{O}(-1)^{\oplus a}$ be the cokernel of this composite.
Then we have the following exact sequence:
\[
\mathcal{O}(-2)\xrightarrow{\pi} \mathcal{O}(-1)^{\oplus a}\to E_2^{-n+1,n-1}\to 0.
\]
We claim here that $a=0$. Suppose, to the contrary, that $a>0$.
Since $E_2^{-n+1,n-1}$ cannot be isomorphic to $\mathcal{O}(-1)^{\oplus a}$,
the morphism $\pi$ above
is not zero.
Therefore the composite of $\pi$ and some projection
$\mathcal{O}(-1)^{\oplus a}\to \mathcal{O}(-1)$ is not zero,
whose quotient is of the form $\mathcal{O}_H(-1)$ for some hyperplane 
$H$ in $\mathbb{Q}^n$.
Hence $E_2^{-n+1,n-1}$ admits $\mathcal{O}_H(-1)$ as a quotient. This is a contradiction.
Thus $a=0$ and $E_2^{-n+1,n-1}=0$. Moreover we see that $E_2^{-n,n-1}\cong \mathcal{O}(-2)$.
Therefore $\mathcal{E}$ fits in the following exact sequence:
\[
0\to \mathcal{O}(-2)\to \mathcal{O}^{\oplus r+1}\to \mathcal{E}\to 0.
\]
This is Case (10) of Theorem~\ref{Chern2OnQ4}.

\section{The case where $\mathcal{E}|_{\mathbb{Q}^3}$ belongs to Case
(9) of Theorem~\ref{Chern2OnQ3}}\label{Case(9)OfTheoremChern2OnQ3}

\subsection{The case $n\geq 4$} 
We will show that this case does not arise. 
Suppose, to the contrary, that there exists a nef vector bundle $\mathcal{E}$
on $\mathbb{Q}^4$ such that
$\mathcal{E}|_{\mathbb{Q}^3}$ fits in the following exact sequence:
\[
0\to 
\mathcal{O}(-2)\xrightarrow{\alpha}
\mathcal{O}(-1)^{\oplus 4}
\to \mathcal{O}^{\oplus r+3}
\to\mathcal{E}|_{\mathbb{Q}^3}\to 0.
\]
We see that $h^0(\mathcal{E}|_{\mathbb{Q}^3})=r+3$,
that $h^q(\mathcal{E}|_{\mathbb{Q}^3})=0$ for $q>0$,
that $h^1(\mathcal{E}(-1)|_{\mathbb{Q}^3})=1$,
and that $h^q(\mathcal{E}(-1)|_{\mathbb{Q}^3})=0$ for $q\neq 1$.
Note that $\alpha^{\vee}(-1)$ is surjective
and thus the image of $H^0(\alpha^{\vee}(-1))$ has dimension at least four.
Hence $H^0(\alpha^{\vee}(-1))$ is injective,
and its dual $H^3(\alpha(-2)):H^3(\mathcal{O}(-4))\to H^3(\mathcal{O}(-3)^{\oplus 4})$ is surjective.
Therefore we infer that $h^1(\mathcal{E}(-2)|_{\mathbb{Q}^3})=1$
and that $h^q(\mathcal{E}(-2)|_{\mathbb{Q}^3})=0$ for $q\neq 1$.

We will apply to $\mathcal{E}$ 
the Bondal spectral sequence~\eqref{BondalSpectral}.
Since we have \eqref{h^0(E(-2))vanish} and $h^q(\mathcal{E}(-1))=0$ 
for $q>0$ by \eqref{firstvanishing},
we see that $h^0(\mathcal{E}(-1))=0$,
that  $h^2(\mathcal{E}(-2))=1$,
and that $h^q(\mathcal{E}(-2))=0$ unless $q=2$.
Hence $h^0(\mathcal{E})=r+3$
and $h^q(\mathcal{E})=0$ 
for $q>0$ by \eqref{firstvanishing}.
Moreover we see that 
$h^2(\mathcal{E}(-3))=2$
and that $h^q(\mathcal{E}(-3))=0$ unless $q=2$.

We have 
$h^q(\mathcal{S}^{\mp\vee}\otimes \mathcal{E})
=h^{q+1}(\mathcal{S}^{\pm\vee}\otimes \mathcal{E}(-1))$ for any $q$
by \eqref{qq+1joushouForE(-1)}.
Note that 
$h^q(\mathcal{S}^{\vee}\otimes \mathcal{E}|_{\mathbb{Q}^{3}})=0$
for any $q$.
Hence
$h^{q+1}(\mathcal{S}^{\pm\vee}\otimes\mathcal{E}(-1))=
h^{q+1}(\mathcal{S}^{\pm\vee}\otimes\mathcal{E})$ for any $q$
by \eqref{RestrictionForSE},
and thus
$h^q(\mathcal{S}^{\pm\vee}\otimes \mathcal{E})=0$ for any $q$.

Now we see that $\Hom(G,\mathcal{E})\cong S_0^{\oplus r+3}$, 
that $\Ext^q(G,\mathcal{E})=0$
unless $q=0$ or $2$, and 
that $\Ext^2(G,\mathcal{E})$ fits in the following exact sequence:
\[
0\to S_4\to \Ext^2(G,\mathcal{E})\to S_5^{\oplus 2}\to 0.
\]
Therefore Lemma~\ref{S2Arilemma} shows that $E_2^{p,q}=0$ unless $(p,q)=(-3,2)$ or $(0,0)$,
that $E_2^{0,0}\cong \mathcal{O}^{\oplus r+3}$,
and that $E_2^{-3,2}$ fits in the following exact sequence:
\[
0\to \mathcal{O}(-1)^{\oplus 2}\to 
T_{\mathbb{P}^5}(-2)|_{\mathbb{Q}^4}\to E_2^{-3,2}\to 0.
\]
Hence 
$\mathcal{E}$ fits in the following exact sequence:
\[
0\to \mathcal{O}(-1)^{\oplus 2}\to 
T_{\mathbb{P}^5}(-2)|_{\mathbb{Q}^4}\to \mathcal{O}^{\oplus r+3}\to \mathcal{E}\to 0.
\]
Since $T_{\mathbb{P}^5}(-2)|_{\mathbb{Q}^4}$ fits in 
an exact sequence
\[
0\to \mathcal{O}(-2)\to \mathcal{O}(-1)^{\oplus 6}\to T_{\mathbb{P}^5}(-2)|_{\mathbb{Q}^4}\to 0,
\]
we infer that 
$\mathcal{E}$ fits in the following exact sequence:
\[
0\to \mathcal{O}(-1)^{\oplus 2}\oplus\mathcal{O}(-2)\to 
\mathcal{O}(-1)^{\oplus 6}\to \mathcal{O}^{\oplus r+3}\to \mathcal{E}\to 0,
\]
Hence
$\mathcal{E}$ fits in the following exact sequence:
\[
0\to \mathcal{O}(-2)\to 
\mathcal{O}(-1)^{\oplus 4}\to \mathcal{O}^{\oplus r+3}\to \mathcal{E}\to 0.
\]
This contradicts that $\mathcal{E}$ is a vector bundle on $\mathbb{Q}^4$.
Hence this case does not arise.

\bibliographystyle{alpha}

\begin{thebibliography}{PSW92}

\bibitem[AO91]{MR1092580}
Vincenzo Ancona and Giorgio Ottaviani.
\newblock Some applications of beilinson's theorem to projective spaces and
  quadrics.
\newblock {\em Forum Math.}, 3(2):157--176, 1991.

\bibitem[Bon89]{MR992977}
Alexey~I. Bondal.
\newblock Representations of associative algebras and coherent sheaves.
\newblock {\em Izv.\ Akad.\ Nauk SSSR Ser.\ Mat.}, 53(1):25--44, 1989.

\bibitem[Ful98]{fl}
Willian Fulton.
\newblock {\em Intersection theory}, volume~2 of {\em Ergebnisse der Mathematik
  und ihrer Grenzgebiete (3)}.
\newblock Springer-Verlag, Berlin, second edition, 1998.

\bibitem[Huy06]{MR2244106}
Daniel Huybrechts.
\newblock {\em Fourier-{M}ukai transforms in algebraic geometry}.
\newblock Oxford Mathematical Monographs. The Clarendon Press Oxford University
  Press, Oxford, 2006.

\bibitem[Kap88]{MR0939472}
M.~M. Kapranov.
\newblock On the derived categories of coherent sheaves on some homogeneous
  spaces.
\newblock {\em Invent.\ Math.}, 92(3):479--508, 1988.

\bibitem[Laz04]{MR2095472}
Robert Lazarsfeld.
\newblock {\em Positivity in algebraic geometry. {II}. Positivity for vector
  bundles, and multiplier ideals.}, volume~49 of {\em Ergebnisse der Mathematik
  und ihrer Grenzgebiete. 3. Folge. A Series of Modern Surveys in Mathematics}.
\newblock Springer-Verlag, Berlin, 2004.

\bibitem[Ohn22]{MR4453350}
Masahiro Ohno.
\newblock Nef vector bundles on a projective space or a hyperquadric with the
  first {C}hern class small.
\newblock {\em Rend.\ Circ.\ Mat.\ Palermo Series 2}, 71(2):755--781, 2022.

\bibitem[Ohn23]{QuadricThreefoldc1=2}
Masahiro Ohno.
\newblock Nef vector bundles on a quadric threefold with first {C}hern class
  two.
\newblock {\em arXiv:2311.02836}, 2023.

\bibitem[OT14]{MR3275418}
Masahiro Ohno and Hiroyuki Terakawa.
\newblock A spectral sequence and nef vector bundles of the first {C}hern class
  two on hyperquadrics.
\newblock {\em Ann.\ Univ.\ Ferrara Sez. VII Sci.\ Mat.}, 60(2):397--406, 2014.

\bibitem[Ott88]{ot}
Giorgio Ottaviani.
\newblock Spinor bundles on quadrics.
\newblock {\em Trans. Amer. Math. Soc.}, 307(1):301--316, 1988.

\bibitem[PSW92]{pswnef}
Thomas Peternell, Micha{\l} Szurek, and Jaros{\l}aw~A. Wi{\'{s}}niewski.
\newblock Numerically effective vector bundles with small {C}hern classes.
\newblock In K.~Hulek, T.~Peternell, M.~Schneider, and F.-O. Schreyer, editors,
  {\em Complex Algebraic varieties, Proceedinigs, Bayreuth, 1990}, number 1507
  in Lecture Notes in Math., pages 145--156, Berlin, 1992. Springer.

\end{thebibliography}
\newcommand{\noop}[1]{} \newcommand{\noopsort}[1]{}
  \newcommand{\printfirst}[2]{#1} \newcommand{\singleletter}[1]{#1}
  \newcommand{\switchargs}[2]{#2#1}

\end{document}